\documentclass[11pt,twoside]{article}
\usepackage{times}
\usepackage{amsmath,amssymb,amsthm}
\usepackage{enumerate}
\usepackage{cite}
\allowdisplaybreaks

\newcommand{\R}{\mathbb R}
\newcommand{\N}{\mathbb N}
\newcommand{\Z}{\mathbb Z}
\newcommand{\p}{\partial}
\newcommand{\ve}{\varepsilon}
\newcommand{\f}{\frac}

\newcommand{\la}{\lambda}

\newcommand{\al}{\alpha}

\renewcommand{\th}{\theta}

\newcommand{\G}{\Gamma}

\renewcommand{\o}{\omega}

\newcommand{\ds}{\displaystyle}

\allowdisplaybreaks

\newcommand{\crit}{\textup{crit}}
\newcommand{\conf}{\textup{conf}}

\theoremstyle{plain}
\newtheorem{theorem}{Theorem}[section]

\newtheorem{lemma}[theorem]{Lemma}

\newtheorem{conclusion}{Theorem}

\theoremstyle{definition}

\newtheorem*{conjecture}{Conjecture}
\newtheorem*{strconjecture}{Strauss Conjecture}

\theoremstyle{remark}
\newtheorem{remark}{Remark}[section]

\numberwithin{equation}{section}

\title{On the global solution problem for semilinear generalized
  Tricomi equations, I}

\author{Daoyin He$^{1, 2,*}$,\qquad Ingo Witt$^{2,*}$, \qquad Huicheng
  Yin$^{3,}$\footnote{He Daoyin
    (\texttt{daoyin.he@mathematik.uni-goettingen.de}) and Yin Huicheng
    (\texttt{huicheng@} \texttt{nju.edu.cn}) are supported by the NSFC
    (No.~11571177) and by the Priority Academic Program Development of
    Jiangsu Higher Education Institutions. Ingo~Witt
    (\texttt{iwitt@mathematik.uni-goettingen.de}) was partly supported
    by the DFG through the Sino-German project ``Analysis of PDEs and
    Applications."  This research was started when Yin Huicheng was
    visiting the Mathematical Institute of the University of
    G\"ottingen in February-March of 2013.}\vspace{0.5cm}\\ \small 1.
  Department of Mathematics and IMS, Nanjing University, Nanjing
  210093, China.\\ \small 2.  Mathematical Institute, University of
  G\"{o}ttingen, Bunsenstr.~3-5, D-37073 G\"{o}ttingen,
  Germany.\\ \small 3.  School of Mathematical Sciences, Nanjing
  Normal University, Nanjing 210023, China.\\}
\vspace{0.5cm}

\begin{document}

\maketitle
\thispagestyle{empty}

\begin{abstract}
In this paper, we are concerned with the global Cauchy problem for the
semilinear generalized Tricomi equation $\partial_t^2 u-t^m \Delta
u=|u|^p$ with initial data $(u(0,\cdot), \p_t u(0,\cdot))=(u_0, u_1)$,
where $t\ge 0$, $x\in\R^n$ ($n\ge 3$), $m\in\N$, $p>1$, and $u_i\in
C_0^{\infty}(\R^n)$ ($i=0,1$). We show that there exists a critical
exponent $p_{\crit}(m,n)>1$ such that the solution $u$, in general,
blows up in finite time when $1<p<p_{\crit}(m,n)$. We further show
that there exists a conformal exponent $p_{\conf}(m,n)>
p_{\crit}(m,n)$ such that the solution $u$ exists globally when
$p>p_{\conf}(m,n)$ provided that the initial data is small enough. In
case $p_{\crit}(m,n)<p\le p_{\conf}(m,n)$, we will establish global
existence of small data solutions $u$ in a subsequent paper \cite{He}.
\end{abstract}

\noindent
\textbf{Keywords.} Generalized Tricomi equation, critical exponent,
conformal exponent, global existence, blowup, Strichartz
estimate.

\noindent
\textbf{2010 Mathematical Subject Classification} 35L70, 35L65, 35L67.

%% -------------------------------------------------------------------

\section{Introduction}
In this paper, we are concerned with the global existence and blowup
of solutions $u$ of the semilinear generalized Tricomi equation
\begin{equation}\label{equ:original}
\left\{ \enspace
\begin{aligned}
&\partial_t^2 u-t^m \Delta u =|u|^p, && (t, x)\in \R^{n+1}_{+},\\
&u(0,\cdot)=u_0(x), \quad \partial_{t} u(0,\cdot)=u_1(x).
\end{aligned}
\right.
\end{equation}
Here, $t\geq 0$, $x=(x_1,\dots, x_n)\in\R^n$ ($n\ge 3$), $m\in\N$,
$p>1$, and $u_i\in C_0^{\infty}(B(0,M))$ $(i=0,1)$, where $B(0,
M)=\{x\colon |x|<M\}$, and $M>0$. In general, one has only weak
solutions of \eqref{equ:original} since the nonlinear term $|u|^p$ is
not $C^2$ when $1<p<2$. For the local existence and regularity of
solutions~$u$ of \eqref{equ:original} under weaker regularity
assumptions on $(u_0, u_1)$, the reader may consult
\cite{Rua1,Rua2,Rua3,Rua4,Yag2} and the references given therein; here
we shall not discuss this problem.

Our present objective is, for given $m\in\N$ and $n\ge 3$, to
determine a critical exponent $p_{\crit}(m,n)>1$ such that solutions
$u$ of $\eqref{equ:original}$ will, in general, blow up in finite time
when $1<p<p_{\crit}(m,n)$ and a conformal exponent
$p_{\conf}(m,n)>p_{\crit}(m,n)$ with the property that small data
solutions $u$ of \eqref{equ:original} exist globally in time when
$p>p_{\conf}(m,n)$. Global existence of small data solutions $u$ of
\eqref{equ:original} for $p$ in the range $p_{\crit}(m,n) < p\leq
p_{\conf}(m,n)$ will be established in a forthcoming paper \cite{He}.

\medskip

Before we describe the content of this paper in detail, we recall a
number of related results. Firstly, we consider the semilinear wave
equation
\begin{equation}\label{equ:wave}
\left\{ \enspace
\begin{aligned}
  &\partial_t^2 u-\Delta u =|u|^p, && (t,x)\in \R^{n+1}_{+}, \\
  &u(0,\cdot)=u_0(x), \quad \partial_{t} u(0,\cdot)=u_1(x),
\end{aligned}
\right.
\end{equation}
where $p>1$, $n\ge 2$, and $u_i\in C_0^{\infty}(\R^n)$ ($i=0, 1$). Let
$p_1(n)$ denote the positive root of the quadratic equation
\begin{equation}\label{1.3}
\left(n-1\right)p^2-\left(n+1\right)p-2=0.
\end{equation}
Strauss \cite{Strauss} made the following conjecture:

\begin{strconjecture}
If $p>p_1(n)$, then small data solutions of problem $\eqref{equ:wave}$
exist globally. If $1<p<p_1(n)$, then small data solutions of problem
$\eqref{equ:wave}$ blow up in finite time.
\end{strconjecture}

For $1<p\leq p_1(n)$ and non-negative initial data $\left(u_0,
u_1\right)$, blowup for the solution $u$ of \eqref{equ:wave} has been
established, while, for $p > p_1(n)$, global existence of small data solution
$u$ of $\eqref{equ:wave}$ has also been systematically studied (see
\cite{Gls1,Gla1, Gla2,Joh,Gls2,Sch,Sid,Yor,Zhou} and the references
therein).  Especially, in \cite{Gls1} and \cite{Gls2}, one finds a
detailed history of results related to the Strauss Conjecture.

\smallskip

Secondly, we consider the semilinear wave equation with time-dependent
dissipation
\begin{equation}\label{equ:eff}
\left\{ \enspace
\begin{aligned}
&\partial_t^2 u-\Delta u +\f{\mu}{(1+t)^{\al}}\,\p_tu=|u|^p, &&
  (t,x)\in \R^{n+1}_{+},\\
&u(0,\cdot)=u_0(x), \quad \partial_{t} u(0,\cdot)=u_1(x),
\end{aligned}
\right.
\end{equation}
where $\mu>0$, $\al\ge 0$, $p>1$, $n\ge 1$, and $u_i\in
C_0^{\infty}(\R^n)$ ($i=0, 1$). Define the Fujita exponent
$p_2(n)=1+\f{2}{n}$ as in \cite{Fuj}. It follows from well-known
results that for the semi-linear heat equation $\p_t u-\Delta u=|u|^p$
with initial data $u(0, \cdot)=u_0(x)$, small data solution $u$
exists globally if $u_0$ is sufficiently small and $p>p_2(n)$;
otherwise, solutions $u$ will, in general, blow up in finite
time.

As for problem \eqref{equ:eff}, the following result has been
established in a series of papers \cite{Rei1, Rei2,Zhai,Nish,Rei3}:

\begin{conclusion}
\begin{enumerate}[{\rm (i)}]
\item For $0\le\al<1$, if $p>p_2(n)$, then $\eqref{equ:eff}$ has a global
small data solution; if $1<p\le p_2(n)$, the solution $u$ of
$\eqref{equ:eff}$ generally blows up in finite time.
\item For $\al>1$, or $\al=1$ and $0<\mu\ll1$, then the properties of
  problem \eqref{equ:eff} are analogous to those of problem
  \eqref{equ:wave}.
\item For $\al=1$ and $\mu\gg1$, then the
  properties of problem \eqref{equ:eff} are analogous to those of the
  semi-linear heat equation $\p_t u-\Delta u=|u|^p$.
\end{enumerate}
\end{conclusion}

\begin{remark}
Note that for $\al=1$ and $\mu\approx1$, it is still an interesting
open problem to determine explicitly a critical value $p_{c}(n)$ so
that problem \eqref{equ:eff} has global small data solutions for
$p>p_c(n)$, while solutions of \eqref{equ:eff}, in general, blow up in
finite time when $1<p\le p_c(n)$.
\end{remark}

Thirdly, we consider the semilinear generalized Tricomi equation
\begin{equation}\label{equ:Tri}
\left\{ \enspace
\begin{aligned}
&\partial_t^2 u-t^{2k}\Delta u=|u|^p,&& (t,x)\in \R^{n+1}_{+},\\
&u(0,\cdot)=u_0(x), \quad \partial_{t} u(0,\cdot)=u_1(x),
\end{aligned}
\right.
\end{equation}
where $k>\f12$ is a real constant, $p>1$, $n\ge 1$, and $u_i\in
C_0^{\infty}(\R^n)$ $(i=0, 1)$. Note that problems $\eqref{equ:eff}$
and $\eqref{equ:Tri}$ are closely related for large $t>0$. Indeed,
with $T=t^{k+1}/(k+1)$, the equation in $\eqref{equ:Tri}$ becomes
\[
\p_T^2 u-\Delta u+\ds\f{k}{k+1}\f{\p_T u}{T}
=(k+1)^{-\f{2k}{k+1}}T^{-\f{2k}{k+1}}\,|u|^p,
\]
which is essentially the equation
\begin{equation}\label{equ:W}
\p_t^2 u-\Delta u+\f{\mu_k}{1+t}\,\p_t u=C_k(1+t)^{-\f{2k}{k+1}}\,|u|^p
\end{equation}
for large $t>0$ with $\mu_k=\f{k}{k+1}$ and $\ds
C_k=(k+1)^{-\f{2k}{k+1}}$.

Comparing the equations in \eqref{equ:eff} and \eqref{equ:W}, one
realizes that their linear parts are identical. Note that the
coefficient $\mu_k$ in \eqref{equ:W} can be arbitrarily close to $1$
when $k$ is large. In this case, however, it is unknown what the
critical value of the exponent $p$ for problem \eqref{equ:eff} is.
This especially means that the methods of \cite{Rei1, Rei2, Zhai,
  Nish, Rei3} are not applicable for studying problem
\eqref{equ:Tri}.

We now recall some known results concerning problem
\eqref{equ:Tri}. Under the conditions
\begin{equation}\label{equ:ugly}
\left\{ \enspace
\begin{aligned}
\f{(n+1)(p-1)}{p+1}&\le\f{k}{k+1},\\
\left(\f{2}{p-1}-\f{n(k+1)}{p+1}\right)p&\le 1,\\
\f{p+1}{p(p-1)n(k+1)}\le\f{1}{p+1}&\le\f{k+2}{(n+1)(p-1)(k+1)}
\end{aligned}
\right.
\end{equation}
(corresponding to (1.8) and (1.12) of \cite{Yag2} with $\al=p-1$ and
$\beta=\f{2}{p-1}-\f{n(k+1)}{p+1}$) it was shown in \cite[Theorem~1.2]{Yag2}
that problem \eqref{equ:Tri} has a global small data solution
$u\in C([0, \infty), L^{p+1}(\R^n)) \cap C^1([0, \infty), {\mathcal D}'(\R^n))$.
On the other hand, under the conditions $\int_{\R^n} u_1(x)dx>0$ and
\begin{equation}\label{equ:blow}
1<p<\f{(k+1)n+1}{(k+1)n-1},
\end{equation}
it was shown in \cite[Theorem~1.3]{Yag2} that problem \eqref{equ:Tri}
has no global solution $u\in C([0, \infty), L^{p+1}(\R^n))$. Here we
  point out that \eqref{equ:blow} comes from condition (1.15) of
  \cite{Yag2}. In particular, for $n=3$, from $\eqref{equ:ugly}$ and
  $\eqref{equ:blow}$ one has (see also (1.16) of \cite{Yag2}):

\begin{conclusion}\label{thmB}
Let $n=3$.
\begin{enumerate}[{\rm (i)}]
\item If $\f{3k+5+\sqrt{9k^2+42k+33}}{6k+4}<p<\min
  \left\{\f{3k+5}{3k+1}, \f{5k+4}{3k+4}\right\}$, then problem
  \eqref{equ:Tri} admits a global small data solution $u\in C([0,
    \infty), L^{p+1}(\R^3))$.
\item If $1<p<\f{3k+4}{3k+2}$, then, in general, the solution of
  problem \eqref{equ:Tri} will blow up in finite time.
\end{enumerate}
\end{conclusion}

Based on this theorem, Yagdjian \cite{Yag2} put forward the following
conjecture (which corresponds to (1.17) of \cite{Yag2} with
$\al=p-1$):

\begin{conjecture}
Let $n=3$ and $\f{3k+4}{3k+2}\leq p\leq
\f{3k+5+\sqrt{9k^2+42k+33}}{6k+4}$. Then small data solutions of
problem \eqref{equ:Tri} exists globally.
\end{conjecture}

In this and a forthcoming paper \cite{He}, we will systematically
study problem $\eqref{equ:original}$. In particular, our analysis will
show that Yagdjian's conjecture fails in a certain range of $p$.

Let $p_{\crit}(m,n)$ be the positive root $p$ of the quadratic equation
\begin{equation}\label{equ:crit}
\left(\left(m+2\right)\frac{n}{2}-1\right)p^2
+\left(\left(m+2\right)\left(1-\frac{n}{2}\right)-3\right)p-(m+2)=0.
\end{equation}
Note that $p_{\crit}(0,n)=p_1(n)$, see \eqref{1.3}. Then our first
result asserts:

\begin{theorem}[Blow up for $1<p<p_{\crit}(m,n)$]\label{thm1.1}
Let $1<p<p_{\crit}(m,n)$ and suppose that $u_i\ge 0$ and
$u_i\not\equiv 0$ for $i=0, 1$. Then problem \eqref{equ:original} admits
no global solution $u$ with $u\in C([0, \infty), H^1(\R^n))
  \cap C^1([0, \infty), L^2(\R^n))$.
\end{theorem}

\begin{remark}
We have $p_{\crit}(2k,3)=\frac{k+4+\sqrt{25k^2+48k+32}}{6k+4}$,
and it follows from a direct computation that
\begin{equation} \label{equ: Ya}
\f{3k+4}{3k+2}<p_{\crit}(2k,3)<\f{3k+5+\sqrt{9k^2+42k+33}}{6k+4}.
\end{equation}
In particular, our result shows that Yagdjian's conjecture fails for
$\f{3k+4}{3k+2}\leq p < p_{\crit}(2k,3)$.
\end{remark}

Next we discuss the global existence problem for
$\eqref{equ:original}$.  Denote by $N=1+\f{\left(m+2\right)n}{2}$
the homogeneous dimension of the operator $\p_t^2-t^{m}\Delta$. Then
the exponent $p$ leading to a conformally invariant equation in
$\eqref{equ:original}$ is
\begin{equation}\label{equ:conf}
p_{\conf}(m,n)=\f{N+2}{N-2}=\frac{\left(m+2\right)n+6}{\left(m+2\right)n-2}.
\end{equation}

\begin{theorem}[Global existence for $p>p_{\conf}(m,n)$]\label{thm1.2}
Let either $p_{\conf}(m,n)<p\le \frac{(m+2)(n-2)+6}{(m+2)(n-2)-2}$
or $p>\frac{(m+2)(n-2)+6}{(m+2)(n-2)-2}$ and $p$ be an integer,
where in the latter case the nonlinearity $|u|^p$ is replaced with
$\pm\, u^p$. Then there exists a constant $\ve_0>0$ such that problem
\eqref{equ:original} admits a global weak solution $u\in
L^r(\R^{n+1}_+)$ whenever
$\|u_0\|_{H^s}+\|u_1\|_{H^{s-\f{2}{m+2}}}\le\ve_0$, where $
s=\f{n}{2}-\f{4}{(m+2)(p-1)}$ and $
r=\frac{(m+2)n+2}{4}\left(p-1\right)$.
\end{theorem}

\begin{remark}\label{rem1.3}
It holds
\[
p_{\conf}(2k,3)=\frac{3k+6}{3k+2}>\f{3k+5+\sqrt{9k^2+42k+33}}{6k+4}.
\]
So, we have especially improved the upper bound (from $
\min\left\{\f{3k+5}{3k+1}, \f{5k+4}{3k+4}\right\}$ to $\infty$) for
the exponent~$p$ in Theorem~\ref{thmB} of Yagdjian to obtain global
existence for small data solutions of problem $\eqref{equ:Tri}$.
\end{remark}

\begin{remark}\label{rem1.4}
If the initial data $u_0\in H^s(\R^n)$ and $u_1\in H^{s-\f{2}{m+2}}(\R^n)$
with $s\ge 0$ is given, then, by a scaling argument as in
\cite{Gls1}, one can deduce that problem \eqref{equ:original} is
ill-posed for $s<\frac{n}{2}-\frac{4}{(m+2)(p-1)}$. See
\cite{Rua4} for details.
\end{remark}

\begin{remark}\label{rem1.5}
A direct verification shows $p_{\crit}(m,n)<p_{\conf}(m,n)$ when $n\ge
3$.  In a forthcoming paper \cite{He} we shall establish the global
existence of small data solution of \eqref{equ:original} when
$p_{\crit}(m,n)<p\le p_{\conf}(m,n)$.
\end{remark}

\begin{remark}\label{rem1.6}
As in \cite[page 368]{Gls2}, where the semilinear wave equation
\eqref{equ:wave} was studied, when $
p>\frac{(m+2)(n-2)+6}{(m+2)(n-2)-2}$, we also impose additional
restrictions on the exponent~$p$ and the nonlinearity appearing in
\eqref{equ:original}. More specifically, we require that $p$ is an
integer and the nonlinearity is equal to $\pm\, u^p$.
\end{remark}

There is an extensive list of results concerning the Cauchy problem
for both linear and semilinear generalized Tricomi equations. For
instances, for linear generalized Tricomi equations, Barros-Neto and
Gelfand in \cite{Bar} and Yagdjian in \cite{Yag1} computed the
fundamental solution explicitly. More recently, the authors of
\cite{Rua1, Rua2, Rua3, Rua4} established the local existence as well
as the singularity structure of low regularity solutions of the
semilinear equation $\partial_t^2u -t^m\triangle u=f(t,x,u)$ in the
degenerate hyperbolic region and the elliptic-hyperbolic mixed region,
respectively, where $f$ is a $C^1$ function and has compact support
with respect to the variable $x$. Yagdjian \cite{Yag2} obtained a
number of interesting results about the global existence and the
blowup of solutions of problem \eqref{equ:original} when the exponent
$p$ belongs to a certain range. In \cite{Yag2}, however, there is a
gap between the global existence interval and the blowup interval;
moreover, the critical exponent $p_{\crit}(m,n)$ was not
determined there. In this paper and in a forthcoming paper \cite{He},
motivated by the Strauss conjecture, we will systematically study the
blowup problem and the global existence problem for
\eqref{equ:original}.

We now comment on the proofs of Theorem~\ref{thm1.1} and
Theorem~\ref{thm1.2}. To prove Theorem~\ref{thm1.1}, we define the
function $G(t)=\int_{\R^n}u(t,x)\,dx$ as in \cite{Yor} and, by
applying some crucial techniques for the modified Bessel function as
in \cite{Hon, Rua3} and by choosing a good test function, we derive a
Riccati-type ordinary differential inequality for $G(t)$ by a delicate
analysis of \eqref{equ:original}. From this, the blowup result in
Theorem 1.1 can be derived under the positivity assumptions of $u_0$
and $u_1$. To prove the global existence result in
Theorem~\ref{thm1.2}, motivated by \cite{Gls1,Gls2}, where basic
Strichartz estimates were obtained for the linear wave operator, we
are required to establish Strichartz estimates for the generalized
Tricomi operator $\p_t^2-t^m\Delta$.  In this process, a series of
inequalities are derived by applying an explicit formula for solutions
of the linear generalized Tricomi equations and by utilizing some
basic properties of related Fourier integral operators.  Based on the
resulting inequalities and the contraction mapping principle, we
eventually complete the proof of Theorem~\ref{thm1.2}.

\medskip

This paper is organized as follows: In $\S 2$,
the blowup result in Theorem 1.1 is obtained.
In $\S 3$, some basic Strichartz inequalities
are established for the linear generalized Tricomi  operator $\p_t^2-t^m\Delta$.
In $\S 4$, by the results in $\S 3$ and contractible mapping principle, we shall complete the
proof of Theorem 1.2.

%% -------------------------------------------------------------------

\section{Proof of Theorem~\ref{thm1.1}}

In this section, we shall prove blowup in finite time for certain
local solutions $u$ of \eqref{equ:original}. To this end, we introduce
the function $G(t)=\int_{\R^n}u(t,x)\,dx$. By some delicate analysis,
we then obtain a Riccati-type differential inequality for $G(t)$ so
that blowup of $G(t)$ can be deduced from the following result (see
\cite[Lemma~4]{Sid}):

\begin{lemma}\label{lem2.1}
Suppose that $G\in C^2([a,b);\R)$ and, for $a\leq t<b$,
\begin{align}
G(t)&\geq C_0(R+t)^\alpha, \label{equ:2.1}\\
G''(t)&\geq C_1(R+t)^{-q}G(t)^p, \label{equ:2.2}
\end{align}
where $C_0$, $C_1$, and $R$ are some positive constants. Suppose further
that $p>1$, $\alpha\geq 1$, and $\left(p-1\right)\alpha\geq q-2$. Then
$b$ is finite.
\end{lemma}

In view of $\operatorname{supp} u_i\subseteq B(0, M)$ ($i=0,1$) and
the finite propagation speed for solutions of hyperbolic equations,
one has that, for any fixed $t>0$, the support of $u(t,\cdot)$ with
respect to the variable $x$ is contained in the ball $B(0,
M+\phi(t))=\{x\colon |x|<M+\phi(t)\}$, where $\phi(t)=
\f{2}{m+2}\,t^{\f{m+2}{2}}$. Then it follows from an integration by
parts that
\begin{equation*}
G''(t)=\int_{\R^n}|u(t,x)|^p\, dx
\geq \frac{\left|\int_{\R^n}u(t,x)\,dx\right|^p}{
\left(\int_{|x|\leq M+\phi(t)}\,dx\right)^{p-1}}
\geq C(M+t)^{-\frac{m+2}{2}\,n(p-1)}\,|G(t)|^p,
\end{equation*}
which means that $G(t)$ fulfills inequality \eqref{equ:2.2} with $
q=\frac{m+2}{2}\,n\left(p-1\right)$ (once inequality \eqref{equ:2.1}
has been verified demonstrating that $G$ is positive). To establish
\eqref{equ:2.1}, we introduce the following two functions: The first
one is
\begin{equation}\label{equ:2.3}
\varphi(x)=\int_{{\mathbb S}^{n-1}}e^{x\cdot \omega}d\omega,
\end{equation}
which was used in \cite{Yor}, where $\varphi(x)$ is also shown to satisfy
\begin{equation} \label{equ:2.4}
\varphi(x)\sim C_n\,|x|^{-\frac{n-1}{2}}e^{|x|} \quad \text{as $|x|\rightarrow\infty$}.
\end{equation}
The second function is the so-called modified Bessel function
\begin{equation*}
K_\nu(t)=\int_0^\infty e^{-t\cosh{z}}\cosh(\nu z)dz, \quad \nu\in \R,
\end{equation*}
which is a solution of the equation
\[
\left(t^2\frac{d^2}{dt^2}+t\frac{d}{dt}-(t^2+\nu^2)\right)K_\nu(t)=0, \quad t>0.
\]
From \cite[page 24]{Erd2}, we have
\begin{equation}\label{equ:2.5}
  K_\nu(t)=\sqrt{\frac{\pi}{2t}}\,e^{-t}\left(1+O(t^{-1})\right) \quad
  \text{as $t\rightarrow \infty$,}
\end{equation}
provided that $\operatorname{Re}\nu>-1/2$. Set
\begin{equation}\label{equ:2.6}
\lambda(t)=C_{m}\,t^{\frac{1}{2}}K_{\frac{1}{m+2}}\left(\frac{2}{m+2}\,t^{\frac{2}{m+2}}\right),
\quad t>0,
\end{equation}
where the constant $C_{m}>0$ is chosen so that $\lambda(t)$ satisfies
\begin{equation}\label{equ:2.7}
\left\{ \enspace
\begin{aligned}
&\la''(t)-t^m\la(t)=0, && t\geq0 \\
&\la(0)=1, \quad \la(\infty)=0.
\end{aligned}
\right.
\end{equation}

Here is a list of properties of $\lambda(t)$ (see \cite[Lemma~2.1]{Hon}):

\begin{lemma}\label{lem2.2}
\textup{(i)} $\la(t)$ and $-\la'(t)$ are both decreasing, moreover,
$\ds \lim_{t\to \infty}\lambda(t)=\lim_{t\to \infty}\lambda'(t)=0$.

\textup{(ii)} There exists a constant $C>1$ such that
\begin{equation} \label{equ:2.8}
\frac{1}{C}\leq\frac{|\lambda'(t)|}{\lambda(t)\,t^{\frac{m}{2}}} \quad
\textup{for $t>0$ \quad and} \quad
\frac{|\lambda'(t)|}{\lambda(t)\,t^{\frac{m}{2}}}\leq C \quad \textup{for
$t\geq1$.}
\end{equation}
\end{lemma}

We now introduce the test function $\psi$ with
\begin{equation}\label{equ:2.9}
\psi(t,x)=\lambda(t)\varphi(x),
\end{equation}
where the definition of $\varphi$ has been given in \eqref{equ:2.3}.
Let
\begin{equation}\label{equ:2.10}
G_1(t)=\int_{\R^n}u(t,x)\psi(t,x)\,dx.
\end{equation}
Then
\begin{equation}\label{equ:2.11}
G''(t)=\int_{\R^n}|u(t,x)|^p\, dx \geq
\frac{|G_1(t)|^p}{\left(\int_{|x|\leq
    M+\phi(t)}\psi(t,x)^{\frac{p}{p-1}}\,dx\right)^{p-1}}.
\end{equation}

For the function $G_1(t)$, we have:

\begin{lemma}\label{lem2.3}
Under the assumptions of\/ \textup{Theorem~\ref{thm1.1}}, there exists
a $t_0>0$ such that
\begin{equation}\label{equ:2.12}
G_1(t)\geq C\,t^{-\frac{m}{2}}, \quad t\geq t_0.
\end{equation}
\end{lemma}
\begin{proof}
In view of $u\in C([0,T),H^1(\R^n))$, one has that
$G_1(t)$ is a continuous function of $t$. In view of $u_0\ge 0$ and $u_0\not\equiv 0$,
we have
\[
G_1(0)=\int_{\R^n}u_0(x)\varphi(x)\,dx\geq c_0,
\]
where $c_0$ is a positive constant. Hence, there exists a constant
$t_1>0$ such that, for $0\leq t\leq t_1$,
\[
G_1(t)\geq\frac{c_0}{2}.
\]
Similarly, by Lemma~\ref{lem2.2} (i) and $u_1\ge 0$ with $u_1\not\equiv 0$, one can also
choose a constant $t_2>0$ such that, for $0\leq t\leq t_2$,
\begin{equation*}
\int_{\R^n}\p_tu(t,x)\psi(t,x)\,dx\geq\frac{c_0}{2}>0.
\end{equation*}
Moreover, by the smoothness of $\lambda(t)$ and $\lambda(0)=1$, we can
find a $t_3>0$ such that
\[
t_3^{\frac{m}{2}}\lambda(t_3)\geq c_1,
\]
where $c_1>0$ is some positive constant. Together with (i) and (ii) of
Lemma~\ref{lem2.2}, this yields, for $0\leq t\leq t_3$,
\[
-\lambda'(t)\geq-\lambda'(t_3)=|\lambda'(t_3)|\geq
Ct_3^{\frac{m}{2}}\lambda(t_3)\geq Cc_1.
\]
Then, by the assumption that $u_0\ge 0$ but $u_0\not\equiv 0$, we have
that, for $0\leq t\leq t_3$,
\begin{equation*}
\int_{\R^n}\left(-\partial_t\psi(t,x)u(t,x)\right)dx\geq\frac{c_2}{2}>0,
\end{equation*}
where $c_2$ is a positive constant. Note that
\[
\Delta_x\left(\int_{{\mathbb S}^{n-1}}e^{x\cdot
  \omega}\,d\omega\right)=\int_{{\mathbb S}^{n-1}} \sum_{i=1}^n
\omega_i^2e^{x\cdot \omega}\,d\omega=\int_{{\mathbb S}^{n-1}}e^{x\cdot
  \omega}\,d\omega.
\]
Let $t_4=\min\{t_1,t_2,t_3\}>0$. Then it follows from a direct
computation that, for $t>t_4$,
%%and $t_4\leq s\leq t$,
\begin{align*}
\int_{t_4}^t\int_{\R^n}|u|^p\psi\, dxds
&=\int_{t_4}^t\int_{\R^n}\left(\partial_s^2u-s^m\Delta u\right)\psi\, dxds \\
&=\int_{\R^n}\left(\psi\partial_s u-u\partial_s\psi\right)dx\biggr|_{s=t}
-\int_{\R^n}\left(\psi\partial_s u-u\partial_s\right)dx\biggr|_{s=t_4},
\end{align*}
which leads to
\begin{equation*}
\int_{\R^n}\left(\psi\partial_s u-u\partial_s\psi\right)dx\biggr|_{s=t}
\geq\int_{\R^n}\left(\psi\partial_s u-u\partial_s\psi\right)dx\biggr|_{s=t_4}
\geq c\equiv \frac{c_0}{2}+\frac{c_2}{2}.
\end{equation*}
This also yields
\begin{equation}\label{equ:2.13}
\begin{aligned}
G'_1(t)-2\lambda'(t)\int_{\R^n}u\varphi\, dx
&=\frac{d}{dt}\left(\int_{\R^n}u\psi\, dx\right)-2\int_{\R^n}u\partial_t\psi\, dx \\
&=\int_{\R^n}\left(\psi\partial_s u-u\partial_s\psi\right)dx\biggr|_{s=t} \geq c.
\end{aligned}
\end{equation}
Now assume that there is a constant $t_5>t_4$ such $G_1(t_5)=0$, but
$G_1(t)>0$ for $t_4\leq t<t_5$.  Then, for $t_4\leq t\leq t_5$,
\[
\lambda(t)\int_{\R^n}u(t,x)\varphi(x)\,dx=\int_{\R^n}u(t,x)\psi(t,x)\,dx=G_1(t)\geq 0.
\]
Together with Lemma 2.2 (i), this yields that for $t_4\leq t\leq t_5$,
\[
\int_{\R^n}u(t,x)\varphi(x)\,dx\geq 0.
\]
Furthermore, by Lemma~\ref{lem2.2} (ii), one has
\[
-\lambda'(t)=|\lambda'(t)|\leq C\lambda(t)t^{\frac{m}{2}}.
\]
Together with (2.17), this yields
\begin{equation}\label{equ:2.14}
G'_1(t)+Ct^{\frac{m}{2}}G_1(t) \geq
G'_1(t)-2\lambda'(t)\int_{\R^n}u\varphi\, dx \geq c.
\end{equation}
Without loss of generality, we can assume that $c=1$ in
\eqref{equ:2.14}. Then, by solving \eqref{equ:2.14}, we get that, for
$t_4\leq t\leq t_5$,
\begin{equation}\label{equ:2.15}
e^{C\phi(t)}G_1(t)\geq
e^{C\phi(t_4)}G_1(t_4)+\frac{t^{-\frac{m}{2}}}{C}
\left(e^{C\phi(t)}-e^{C\phi(t_4)}\right).
\end{equation}
Therefore, $G_1(t_5)>0$ holds which is a contradiction to $G_1(t_5)=0$.

Thus, we have that, for all $t\ge t_4$,
\[
G_1(t)>0.
\]
Using Lemma~\ref{lem2.2} (ii) again and repeating the argument from
above, one easily obtains the existence of a uniform positive constant
$\tilde{C}$ such that for  $t\ge t_4$
\[
G_1(t)\geq \tilde{C}\,t^{-\frac{m}{2}}.
\]
This proves Lemma~\ref{lem2.3}.
\end{proof}

Relying on Lemma~\ref{lem2.3}, we are now able to prove
Theorem~\ref{thm1.1}.

\begin{proof}[Proof of Theorem~\ref{thm1.1}]
By \eqref{equ:2.5} and \eqref{equ:2.6}, we have that
\[
\lambda(t)\sim t^{-\frac{m}{4}}e^{-\phi(t)} \quad \text{as $t\rightarrow \infty$.}
\]
Next we estimate the denominator $\left(\int_{|x|\leq
  M+\phi(t)}\psi(t,x)^{\frac{p}{p-1}} \, dx\right)^{p-1}$ in
\eqref{equ:2.11}. Note that
\begin{equation*}
\left(\int_{|x|\leq
  M+\phi(t)}\psi(t,x)^{\frac{p}{p-1}}\,dx\right)^{p-1}
=\lambda(t)^p\left(\int_{|x|\leq
  M+\phi(t)}\varphi(x)^{\frac{p}{p-1}}\,dx\right)^{p-1}
\end{equation*}
and
\[
|\varphi(x)|\leq C_n\left(1+|x|\right)^{-\frac{n-1}{2}}e^{|x|}.
\]
Then
\begin{multline*}
\int_{|x|\leq M+\phi(t)}\varphi(x)^{\frac{p}{p-1}}\,dx \\
\begin{aligned}
&\le C\int_0^{\frac{M+\phi(t)}{2}}(1+r)^{n-1-\frac{n-1}{2}\cdot\frac{p}{p-1}}e^{\frac{p}{p-1}r}\,dr
+C\int_{\frac{M+\phi(t)}{2}}^{M+\phi(t)}(1+r)^{n-1-\frac{n-1}{2}\cdot\frac{p}{p-1}}e^{\frac{p}{p-1}r}\,dr \\
&\leq C e^{\frac{M+\phi(t)}{2}}+\left(M+\phi(t)\right)^{n-1-\frac{n-1}{2}\cdot\frac{p}{p-1}}
e^{p\left(M+\phi(t)\right)} \\
&\leq C\left(M+\phi(t)\right)^{n-1-\frac{n-1}{2}\cdot\frac{p}{p-1}}
e^{p\left(M+\phi(t)\right)}
\end{aligned}
\end{multline*}
and
\begin{equation}\label{equ:2.16}
\begin{aligned}
\left(\int_{|x|\leq M+\phi(t)}\psi(t,x)^{\frac{p}{p-1}}\,dx\right)^{p-1}
&\leq Ct^{-\frac{m}{4}p}e^{-p\phi(t)}\left(M+\phi(t)\right)^{(n-1)(p-1)-\frac{n-1}{2}p}
e^{p\left(M+\phi(t)\right)} \\
&\leq Ct^{-\frac{m}{4}p}\left(M+\phi(t)\right)^{(n-1)(p-1)-\frac{n-1}{2}p}.
\end{aligned}
\end{equation}
Therefore, it follows from \eqref{equ:2.11} and \eqref{equ:2.16} that, for $t\ge t_0$,
\begin{equation}\label{equ:2.17}
G''(t)\geq ct^{-\frac{m}{4}p}\left(M+\phi(t)\right)^{\frac{n-1}{2}p-(n-1)(p-1)}
\ge Ct^{\frac{p}{2}}\left(M+\phi(t)\right)^{n-1-\frac{n}{2}p}.
\end{equation}
Integrating \eqref{equ:2.17} twice gives
\[
G(t)\geq C\left(M+t\right)^{\frac{p}{2}+2+\frac{m+2}{2}\left(n-1-\frac{n}{2}\,p\right)}
+C_1\left(t-t_0\right)+C_2.
\]
Note that if
\begin{equation}\label{equ:2.18}
\frac{p}{2}+2+\frac{m+2}{2}\left(n-1-\frac{n}{2}\,p\right)>1
\end{equation}
holds, then one has, for $t\ge t_0$,
\begin{equation}
G(t)\geq C\left(M+t\right)^{\frac{p}{2}+2+\frac{m+2}{2}
  \left(n-1-\frac{n}{2}\,p\right)}.
\end{equation}
This means that condition \eqref{equ:2.1} holds with
$\al=\frac{p}{2}+2+\frac{m+2}{2}(n-1-\frac{n}{2}p)$.

To conclude the proof of Theorem~\ref{thm1.1} we now apply
Lemma~\ref{lem2.1}.  For $n\geq3$, one easily checks that all
$p<p_{\conf}(m,n)$ satisfy \eqref{equ:2.18}.  On the other hand, if we
take
\[
\alpha=\frac{p}{2}+2+\frac{m+2}{2}\left(n-1-\frac{n}{2}\,p\right), \quad
q=\frac{m+2}{2}\,n\left(p-1\right),
\]
then the condition $(p-1)\alpha>q-2$ in Lemma~\ref{lem2.1} becomes
\[
(p-1)\left(\frac{p}{2}+2+\frac{m+2}{2}\left(n-1-\frac{n}{2}\,p\right)\right)
>\frac{m+2}{2}\,n\left(p-1\right)-2,
\]
which is equivalent to
\[
\left((m+2)\,\frac{n}{2}-1\right)p^2+\left((m+2)\left(1-\frac{n}{2}\right)
-3\right)p-(m+2)<0.
\]
The latter means that
\[
p<p_{\crit}(m,n) =\frac{\left(\frac{n}{2}-1\right)(m+2)
  +3+\sqrt{\left(\frac{n^2}{4}+n+1\right)(m+2)^2+(3n-10)(m+2)
    +9}}{\left(m+2\right)n-2}.
\]
By a direct verification, we have that $p_{\crit}(m,n)$ satisfies
\eqref{equ:crit} and that $p_{\crit}(m,n)<p_{\conf}(m,n)$ holds.

We complete the proof of Theorem~\ref{thm1.1} by appealing to
Lemma~\ref{lem2.1} with $a=t_0$ and $b=t$.
\end{proof}

%% -------------------------------------------------------------------

\section{Strichartz estimates for the generalized Tricomi operator}\label{sec3}

Before establishing Strichartz estimates for the generalized Tricomi
operator, we recall two results from \cite[Lemma~3.8]{Gv} and
\cite[Theorem~1.2]{Chr1}.

\begin{lemma}\label{lem3.1}
Let $\ds \beta\in C_0^\infty\left((1/2,2)\right)$ and
$\ds \sum\limits_{j=-\infty}^\infty\beta\left(2^{-j}\tau\right) \equiv1$ for
$\tau>0$.  Define the Littlewood-Paley operators as
\[
G_j(t,x)=(2\pi)^{-n}\int_{\R^n}e^{ix\cdot\xi}\beta\left(2^{-j}|\xi|\right)
\hat{G}(t,\xi)\,d\xi, \quad j\in \Z.
\]
Then
\begin{align*}
\| G\|_{L^s_tL^q_x} \leq C\left(\sum\limits_{j=-\infty}^{\infty}\| G_j
\|^2_{L^s_tL^q_x}\right)^{1/2},& \quad
2\leq q<\infty,\, 2\leq s \leq \infty, \\
\intertext{and}
\left(\sum\limits_{j=-\infty}^{\infty}\| G_j\|^2_{L^r_tL^p_x}\right)^{1/2}
\leq C\| G\|_{L^r_tL^p_x},& \quad 1<p\leq2,\, 1\leq r \leq 2.
\end{align*}
\end{lemma}

\begin{lemma}\label{lem3.2}
Suppose that $1\leq p<q\leq\infty$. Let $T: L^p(\R)\rightarrow
L^q(\R)$ be a bounded linear operator which is defined by
\[
Tf(x)=\int_{\R}K(x,y)f(y)dy,
\]
where $K$ is locally integrable. Define
\[
\tilde{T}f(x)=\int_{-\infty}^x K(x,y)f(y)dy.
\]
Then
\[
\|\tilde{T}f\|_{L^q}\leq C_{p,q}\,\| T\|_{L^p\to L^q}\,\| f\|_{L^p}.
\]
\end{lemma}

In order to prove Theorem~\ref{thm1.2}, we need to establish
Strichartz estimates for the operator $\partial_t^2-t^m\triangle$. To
this end, we study the linear Cauchy problem
\begin{equation}\label{equ:3.1}
\left\{ \enspace
\begin{aligned}
&\partial_t^2 u-t^m\triangle u=F(t,x), && (t,x)\in\R_+^{n+1},\\
&u(0,\cdot)=f(x),\quad \partial_tu(0,\cdot)=g(x).
\end{aligned}
\right.
\end{equation}
Note that the solution $u$ of \eqref{equ:3.1} can be written as
\[
u(t,x)=v(t,x)+w(t,x),
\]
where $v$ solves the homogeneous problem
\begin{equation}\label{equ:3.2}
\left\{ \enspace
\begin{aligned}
&\partial_t^2 v-t^m\triangle v=0, && (t,x)\in\R_+^{n+1},\\
&v(0,\cdot)=f(x),\quad \partial_tv(0,\cdot)=g(x).
\end{aligned}
\right.
\end{equation}
and $w$ solves the inhomogeneous problem with zero initial data
\begin{equation}\label{equ:3.3}
\left\{ \enspace
\begin{aligned}
&\partial_t^2 w-t^m\triangle w=F(t,x), && (t,x)\in\R_+^{n+1},\\
&w(0,\cdot)=0,\quad \partial_tw(0,\cdot)=0.
\end{aligned}
\right.
\end{equation}
Let $\dot{H}^s(\R^n)$ denote the homogeneous Sobolev space with norm
\[
\| f\|_{\dot{H}^s(\R^n)}=\left\| |D_x|^s f\right\|_{L^2(\R^n)},
\]
where
\[
|D_x|=\sqrt{-\Delta}.
\]
If $g \equiv 0$ in \eqref{equ:3.2}, we intend to establish
the Strichartz-type inequality
\[
\| v\|_{L^q_tL^r_x}\leq C\,\| f\|_{\dot{H}^s(\R^n)},
\]
where $q\ge 1$ and $r\ge 1$ are suitable constants related to $s$.
One obtains by a scaling argument that those indices should satisfy
\begin{equation}\label{equ:3.4}
  \frac{1}{q}+\frac{m+2}{2}\cdot\frac{n}{r}
  =\frac{m+2}{2}\left(\frac{n}{2}-s\right).
\end{equation}
Setting $r=q$ and $s=\frac{1}{m+2}$ in \eqref{equ:3.4}, we find that
\begin{equation}\label{equ:3.5}
q=q_0\equiv \frac{2((m+2)n+2)}{(m+2)n-2}>1, \quad n\ge 2,\,m\in\N.
\end{equation}
Note that problem~\eqref{equ:original} is ill-posed for
$u_0\in H^s(\R^n)$ with $s<\f{n}{2}-\f{4}{(m+2)(p-1)}$ (see
Remark~\ref{rem1.4}), while $p\ge p_{\conf}(m,n)$ and
$s=\f{n}{2}-\f{4}{(m+2)(p-1)}$ imply $s\ge\f{1}{m+2}$.

We now prove:

\begin{lemma}\label{lem3.3}
Let $n\geq2$ and $v$ solve problem \eqref{equ:3.2}. Further let $
\frac{1}{m+2}\le s<\f{n}{2}$. Then
\begin{equation}\label{equ:3.6}
\| v\|_{L^q(\R^{n+1}_+)} \leq C\left(\| f\|_{\dot{H}^s(\R^n)}+\|
g\|_{\dot{H}^{s-\frac{2}{m+2}}(\R^n)}\right),
\end{equation}
where $q=\f{2((m+2)n+2)}{(m+2)(n-2s)}\ge q_0$ and the constant
$C>0$ only depends on $m$, $n$, and $s$.
\end{lemma}
\begin{proof}
It follows from \cite{Yag2} that the solution $v$ of \eqref{equ:3.2}
can be written as
\[
v(t,x)=V_1(t, D_x)f(x)+V_2(t, D_x)g(x),
\]
where the operators $V_j(t, D_x)$ ($j=1,2$) have the symbols $V_j(t,
\xi)$ given by
\begin{multline}\label{equ:3.7}
  V_1(t,\xi)=V_1(t,|\xi|)=\frac{\Gamma\bigl(\frac{m}{m+2}\bigr)}{
    \Gamma\bigl(\frac{m}{2(m+2)}\bigr)}\,e^{\frac{z}{2}}\,
  H_+\left(\frac{m}{2(m+2)},\frac{m}{m+2};z\right) \\
  +\,\frac{\Gamma\bigl(\frac{m}{m+2}\bigr)}{
    \Gamma\bigl(\frac{m}{2(m+2)}\bigr)}\,e^{-\frac{z}{2}}\,
  H_-\left(\frac{m}{2(m+2)},\frac{m}{m+2};z\right)
\end{multline}
and
\begin{multline}\label{equ:3.8}
  V_2(t,\xi)=V_2(t,|\xi|)=t\,\frac{\Gamma\bigl(\frac{m+4}{m+2}\bigr)}{
    \Gamma\bigl(\frac{m+4}{2(m+2)}\bigr)}\,e^{\frac{z}{2}}\,
  H_+\left(\frac{m+4}{2(m+2)},\frac{m+4}{m+2};z\right) \\
  +\,t\,\frac{\Gamma\bigl(\frac{m+4}{m+2}\bigr)}{
    \Gamma\bigl(\frac{m+4}{2(m+2)}\bigr)}\,e^{-\frac{z}{2}}\,
  H_-\left(\frac{m+4}{2(m+2)},\frac{m+4}{m+2};z\right),
\end{multline}
where $z=2i\phi(t)|\xi|$.
%%with $i=\sqrt{-1}$, $\xi=(\xi_1,\dots, \xi_n)$, and
%%$|\xi|=\sqrt{\xi_1^2+\dot\cdot\cdot+\xi_n^2}$.
For $\al,\, \nu\in\R$, $\o\in\mathbb C$, we have
\begin{align*}
  H_+(\al,\nu; \o)&=\f{e^{-i\pi(\nu-\al)}}{e^{i\pi(\nu-\al)}-e^{-i\pi(\nu-\al)}}\,
  \f{1}{\G(\nu-\al)}\,\o^{\al-\nu}
\int_{\infty}^{(0+)}e^{-\th}\th^{\nu-\al-1}\left(1-\f{\th}{\o}\right)^{\al-1}\,d\th,\\
H_-(\al,\nu; \o)&=\f{1}{e^{i\pi\al}-e^{-i\pi\al}}\, \f{1}{\G(\al)}\, \o^{-\al}
\int_{\infty}^{(0+)}e^{-\th}\th^{\al-1}\left(1+\f{\th}{\o}\right)^{\nu-\al-1}\,d\th.
\end{align*}
By \cite[Section~3]{Yag2}, one has that, for $\phi(t)|\xi|\geq 1$,
\begin{align}
  \bigl| \partial_\xi^\beta H_{+}\left(\alpha,\gamma;2i\phi(t)|\xi|\right)
  \bigr|&\leq C\left(\phi(t)|\xi|\right)^{\alpha-\gamma}
  \left(1+|\xi|^2\right)^{-\f{|\beta|}{2}},
\label{equ:3.9} \\
\bigl| \partial_\xi^\beta H_{-}\left(\alpha,\gamma;2i\phi(t)|\xi|\right)\bigr|
&\leq C\left(\phi(t)|\xi|\right)^{-\alpha}
\left(1+|\xi|^2\right)^{-\f{|\beta|}{2}}. \label{equ:3.10}
\end{align}
We only estimate $V_1$, since estimating $V_2$ is similar. Indeed, up
to a factor of $t\,\phi(t)^{-\frac{m+4}{2(m+2)}}=
C_m\phi(t)^{-\frac{m}{2(m+2)}}$, the powers of $t$ appearing in $V_1$
or $V_2$ are the same.

Choose $\chi\in C^\infty(\R_+)$ such that
\begin{equation}\label{equ:3.11}
\chi(s)=
\begin{cases}
1, &s\geq2, \\
0, &s\leq1.
\end{cases}
\end{equation}
Then
\begin{equation}\label{equ:3.12}
\begin{aligned}
  V_1(t,|\xi|)\hat{f}(\xi)&=\chi(\phi(t)|\xi|)V_1(t,|\xi|)\hat{f}(\xi)
  +(1-\chi(\phi(t)|\xi|))V_1(t,|\xi|)\hat{f}(\xi)
\\ &\equiv\hat{v}_1(t,\xi)+\hat{v}_2(t,\xi).
\end{aligned}
\end{equation}
Using \eqref{equ:3.7}, \eqref{equ:3.9}, and \eqref{equ:3.10}, we
derive that
\begin{equation}\label{equ:3.13}
{v}_1(t,x)=C_m\left(\int_{\R^n}e^{i\left(x\cdot\xi+\phi(t)|\xi|\right)}
a_{11}(t,\xi)\hat{f}(\xi)\,d\xi+
\int_{\R^n}e^{i\left(x\cdot\xi-\phi(t)|\xi|\right)}
a_{12}(t,\xi)\hat{f}(\xi)\,d\xi\right),
\end{equation}
where $C_m>0$ is a constant only depending on $m$, and, for $l=1,2$,
\begin{equation*}
\bigl| \partial_\xi^\beta a_{1l}(t,\xi)\bigr|\leq
C_{l\beta}\,|\xi|^{-|\beta|}\left(\phi(t)|\xi|\right)^{-\f{m}{2(m+2)}}.
\end{equation*}
On the other hand, it follows from \cite{Erd1} that
\begin{equation}\label{equ:3.14}
V_1(t,|\xi|)=e^{-\frac{z}{2}}\,\Phi\left(\frac{m}{2(m+2)},\frac{m}{m+2};z\right),
\end{equation}
where $\Phi$ is the confluent hypergeometric functions which is
analytic with respect to the variable $z=2i\phi(t)|\xi|$.
%Moreover, for sufficiently large $|z|$, one has the estimate
%\begin{equation*}
%  \left|\Phi\left(\frac{m}{2(m+2)},\frac{m}{m+2};z\right)\right|\leq
%  C\,|z|^{-\frac{m}{2(m+2)}}\left(1+O\left(|z|^{-1}\right)\right).
%\end{equation*}
Then
\begin{equation*}
\left|\partial_\xi\bigl((1-\chi(\phi(t)|\xi|))V_1(t,|\xi|)\bigr)\right|\leq
C\left(1+\phi(t)|\xi|\right)^{-\f{m}{2(m+2)}}|\xi|^{-1}.
\end{equation*}
Similarly, one has
\begin{equation*}
  \left|\partial_\xi^{\beta}\bigl((1-\chi(\phi(t)|\xi|))V_1(t,|\xi|)\bigr)
  \right|\leq
C\left(1+\phi(t)|\xi|\right)^{-\f{m}{2(m+2)}}|\xi|^{-|\beta|}.
\end{equation*}
Thus, we arrive at
\begin{equation}\label{equ:3.15}
v_2(t,x)=C_m\left(\int_{\R^n}e^{i\left(x\cdot\xi+\phi(t)|\xi|\right)}
a_{21}(t,\xi)\hat{f}(\xi)\,d\xi +\int_{\R^n}
e^{i\left(x\cdot\xi-\phi(t)|\xi|\right)}
a_{22}(t,\xi)\hat{f}(\xi)\,d\xi\right),
\end{equation}
where, for $l=1,2$,
\begin{equation*}
\bigl| \partial_\xi^\beta a_{2l}(t,\xi)\bigr|\leq
C_{l\beta}\left(1+\phi(t)|\xi|\right)^{-\f{m}{2(m+2)}}|\xi|^{-|\beta|}.
\end{equation*}
Substituting \eqref{equ:3.13} and \eqref{equ:3.15} into
\eqref{equ:3.12} yields
\[
V_1(t,D_x)f(x)=C_m\left(\int_{\R^n}
e^{i\left(x\cdot\xi+\phi(t)|\xi|\right)} a_1(t,\xi)\hat{f}(\xi)\,d\xi
+\int_{\R^n} e^{i\left(x\cdot\xi-\phi(t)|\xi|\right)}a_2(t,\xi)
\hat{f}(\xi)\,d\xi\right),
\]
where the $a_l$ ($l=1,2$) satisfy
\begin{equation}\label{equ:3.16}
  \bigl|\partial_\xi^\beta a_l(t,\xi)\bigr|\leq
  C_{l\beta}\left(1+\phi(t)|\xi|\right)^{-\f{m}{2(m+2)}}|\xi|^{-|\beta|}.
\end{equation}
We only treat the integral
$\int_{\R^n}e^{i\left(x\cdot\xi+\phi(t)|\xi|\right)}
a_1(t,\xi)\hat{f}(\xi)\,d\xi$, since the treatment of the
integral \linebreak
$\int_{\R^n}e^{i\left(x\cdot\xi-\phi(t)|\xi|\right)} a_2(t,\xi)
\hat{f}(\xi)\,d\xi$ is similar. Denote
\begin{equation}\label{equ:3.17}
(Af)(t,x)=\int_{\R^n}e^{i\left(x\cdot\xi+\phi(t)|\xi|\right)}a_1(t,\xi)
  \hat{f}(\xi)\,d\xi.
\end{equation}
We will show that
\begin{equation} \label{equ:3.18}
  \|(Af)(t,x)\|_{L^q(\R^{n+1}_+)}\le C\,\| f\|_{\dot{H}^s(\R^n)}.
\end{equation}
Note that if we set
\[
\tilde{a}(t,\xi)=\frac{a_1(t,\xi)}{|\xi|^s}, \quad
\hat{h}(\xi)=|\xi|^s\hat{f}(\xi),
\]
then \eqref{equ:3.18} is equivalent to
\begin{equation}\label{equ:3.19}
\left\|\int_{\R^n}e^{i\left(x\cdot\xi+\phi(t)|\xi|\right)}
\tilde{a}(t,\xi)\hat{h}(\xi)\,d\xi\right\|_{L^q(\R^{n+1}_+)}\le
C\,\| h\|_{L^2(\R^n)}.
\end{equation}
We denote the the integral operator in the left-hand side of
\eqref{equ:3.19} still by $A$. In order to prove \eqref{equ:3.19} it
suffices to establish the its dual version
\begin{equation}\label{equ:3.20}
\| A^*G\|_{L^2(\R^n)}\le C\, \| G\|_{L^p(\R^{n+1}_+)},
\end{equation}
where
\[
(A^*G)(y)=\int_{\R^n}\int_{\R^{n+1}_+}e^{i\left(y-x)\cdot\xi-\phi(t)|\xi|\right)}\,
\overline{\tilde{a}(t,\xi)}\,G(t,x)\,dtdxd\xi
\]
is the adjoint operator of $A$, $\f{1}{p}+\f{1}{q}=1$, and $
1\leq p\leq p_0\equiv\frac{2((m+2)n+2)}{(m+2)n+6}$ (note that
$\f{1}{p_0}+\f{1}{q_0}=1$). In view of
\begin{equation}\label{equ:3.21}
\int_{\R^n}|(A^*G)(y)|^2\,dy =\int_{\R^{n+1}_+}(A
A^*G)(t,x)\overline{G(t,x)}\,dtdx \leq\| A A^*G \|_{L^{q}(\R^{n+1}_+)}\|\,
G\|_{L^p(\R^{n+1}_+)},
\end{equation}
one derives that \eqref{equ:3.20} holds if
\begin{equation}\label{equ:3.22}
\| A A^*G \|_{L^{q}(\R^{n+1}_+)}\le
C\, \| G\|_{L^p(\R^{n+1}_+)}, \quad 1\leq p\leq p_0.
\end{equation}
One can write
\begin{equation}\label{equ:3.23}
  (A A^*G)(t,x)=\int_{\R^{n+1}_+}\int_{\R^n}
  e^{i\left((\phi(t)-\phi(\tau))|\xi|+(x-y)\cdot\xi\right)}\,
  \tilde{a}(t,\xi)\,\overline{\tilde{a}(\tau,\xi)}G(\tau,y)\,d\xi
  d\tau dy.
\end{equation}
If we choose a function $\beta\in
C_0^\infty((1/2,2))$ as in Lemma~\ref{lem3.1}
%such that for $\tau>0$,
%\begin{equation}
%\sum\limits_{j=-\infty}^\infty\beta(\frac{\tau}{2^j}) \equiv1,\label{equ:3.24}
%\end{equation}
and set $a_\lambda(t,\tau,\xi)=\beta(|\xi|/\lambda)\tilde{a}(t,\xi)$
$\overline{\tilde{a}(\tau,\xi)}$ for $\lambda>0$, then we obtain a dyadic
decomposition of the operator $AA^\ast$ by
\begin{equation}\label{equ:3.25}
(AA^*)_\lambda G=\int_{\R^{n+1}_+}\int_{\R^n}
  e^{i\left((\phi(t)-\phi(\tau))|\xi|+(x-y)\cdot\xi\right)}a_\lambda(t,\tau,\xi)
  G(\tau,y)\,d\xi d\tau dy.
\end{equation}
In order to prove \eqref{equ:3.22}, we only need to prove
\begin{equation}\label{equ:3.26}
\|(AA^*)_\lambda G \|_{L^{p'}(\R^{n+1}_+)}\le
C\, \| G\|_{L^p(\R^{n+1}_+)},
\quad 1\leq p\leq p_0,
\end{equation}
with the constant $C>0$ independent of $\lambda>0$. Indeed, if
\eqref{equ:3.26} holds, then it follows from Lemma~\ref{lem3.1} and
$p\leq p_0=\frac{2((m+2)n+2)}{(m+2)n+6}<2$ that
\begin{equation*}
\begin{aligned}
\| AA^* G \|_{L^{q}}^2&\leq C\sum_{j\in\Z}\| (AA^*)_{2^j}G\|_{L^{q}}^2
\leq C\sum_{j\in\Z}\sum_{k:|j-k|\leq C_0}\| (AA^*)_{2^j}G_k\|_{L^{q}}^2 \\
&\leq C\sum_{j\in\Z}\sum_{k:|j-k|\leq C_0}\| G_k\|_{L^{p}}^2
\le C\, \| G\|_{L^p(\R^{n+1}_+)},
\end{aligned}
\end{equation*}
where $\hat{G}_k(\tau,\xi)=\beta(2^{-k}|\xi|)\,\hat{G}(\tau,\xi)$.

Next we prove \eqref{equ:3.26}. We will use interpolation between the
two cases $p=1$ and $p=p_0$.

For $p=1$, a direct analysis shows that
\[
|a_\lambda(t,\tau,\xi)|\leq |\xi|^{-2s}
\]
and
\begin{equation}\label{equ:3.27}
\begin{aligned}
\| (AA^*)_\lambda G \|_{L^\infty(\R^{n+1}_+)}
&\leq \int_{\R^{n+1}_+}\left|\int_{\R^n} e^{i[(\phi(t)-\phi(\tau))|\xi|+(x-y)\cdot\xi]}
a_\lambda(t,\tau,\xi)\,d\xi\right||G(\tau,y)|\, dyd\tau \\
&\leq \int_{\R^{n+1}_+}\left|\int_{\R^n}\beta\left(\frac{|\xi|}{\lambda}\right)
|\xi|^{-2s}\,d\xi\right||G(\tau,y)|\, dyd\tau \\
&\leq C\lambda^{n-2s}\| G\|_{L^1(\R^{n+1}_+)}.
\end{aligned}
\end{equation}
Next we prove the endpoint case $p=p_0$ in \eqref{equ:3.26}. Namely,
we shall show that
\begin{equation}\label{equ:3.28}
\| (AA^*)_\lambda G\|_{L^{q}(\R^{n+1}_+)}\leq
C\lambda^{\frac{2}{m+2}-2s}\left\| G\right\|_{L^{p_0}(\R^{n+1}_+)}.
\end{equation}
Note that, for any $t,\tau\in\R_+$ and $\bar{t}=\max\{t,\tau\}$, one
has that
\begin{equation}\label{equ:3.29}
\left|\partial_\xi^\beta\Bigl(\bar{t}^{\frac{m}{(m+2)n+2}}
a_\lambda(t,\tau,\xi)\Bigr)\right|\leq
|\xi|^{-2s-\frac{2m}{(m+2)((m+2)n+2)}-|\beta|}.
\end{equation}
Indeed, without loss of generality, one can assume that
$t\geq\tau$. Then it follows from \eqref{equ:3.16} and a direct
computation that
\begin{equation*}
\begin{aligned}
\left|\partial_\xi^\beta\Bigl(\bar{t}^{\frac{m}{(m+2)n+2}}a_\lambda(t,\tau,\xi)\Bigr)
\right|
&\leq t^{\frac{m}{(m+2)n+2}}\left(1+\phi(t)|\xi|\right)^{-\f{m}{2(m+2)}}
\left(1+\phi(\tau)|\xi|\right)^{-\f{m}{2(m+2)}}|\xi|^{-|\beta|-2s} \\
&\leq \phi(t)^{\frac{2m}{(m+2)((m+2)n+2)}}\left(\phi(t)
|\xi|\right)^{-\frac{2m}{(m+2)((m+2)n+2)}} |\xi|^{-|\beta|-2s} \\
&\leq |\xi|^{-2s-\frac{2m}{(m+2)((m+2)n+2)}-|\beta|}.
\end{aligned}
\end{equation*}
Set
\begin{equation*}
b(t,\tau,\xi)=\lambda^{2s+\frac{2m}{(m+2)((m+2)n+2)}}\bar{t}^{\frac{m}{(m+2)n+2}}
a_\lambda(t,\tau,\xi).
\end{equation*}
Then
\[
\bigl|\partial_\xi^\beta b(t,\tau,\xi)\bigr|\leq |\xi|^{-|\beta|}
\]
and we can write
\begin{multline*}
(AA^*)_\lambda G=\int_{\R^{n+1}_+}\int_{\R^n}
  e^{i\left((\phi(t)-\phi(\tau))|\xi|+(x-y)\cdot\xi\right)}
  \bar{t}^{\,-\frac{m}{(m+2)n+2}}\lambda^{-2s-\frac{2m}{(m+2)\left((m+2)n+2\right)}}
  \\ \times b(t,\tau,\xi)G(\tau,y)\,d\xi dyd\tau.
\end{multline*}
Introduce the operator
\[
T_{t,\tau}f(x)=\int\int e^{i\left((\phi(t)-\phi(\tau))|\xi|+(x-y)\cdot\xi\right)}
\bar{t}^{\,-\frac{m}{(m+2)n+2}}b(t,\tau,\xi)f(y)\,d\xi dy.
\]
Then, by $\max\{t,\tau\}\geq|t-\tau|$, we have that
\begin{equation}\label{equ:3.30}
\| T_{t,\tau}f\|_{L^2(\R^n)}\leq C\left|t-\tau\right|^{-\frac{m}{(m+2)n+2}}
\left\| f\right\|_{L^2(\R^n)}.
\end{equation}
On the other hand, it follows from the method of stationary phase that
\begin{equation}\label{equ:3.31}
\begin{aligned}
\| T_{t,\tau}f\|_{L^\infty(\R^n)}&\leq
C\lambda^{\frac{n+1}{2}}\bar{t}^{\,-\frac{m}{(m+2)n+2}}
\left|\phi(t)-\phi(\tau)\right|^{-\frac{n-1}{2}}\left\| f\right\|_{L^1(\R^n)} \\
&\leq C\lambda^{\frac{n+1}{2}}\left|t-\tau\right|^{-\frac{m}{(m+2)n+2}}
\left|t-\tau\right|^{-\frac{n-1}{2}\cdot\frac{m+2}{2}}\left\| f\right\|_{L^1(\R^n)}.
\end{aligned}
\end{equation}
Together with \eqref{equ:3.30}, this yields
\begin{equation}\label{equ:3.32}
  \| T_{t,\tau}f\|_{L^{q_0}(\R^n)}\leq C\lambda^{\frac{2(n+1)}{(m+2)n+2}}
  \left|t-\tau\right|^{-\frac{(m+2)n-2}{(m+2)n+2}}\left\| f\right\|_{L^{p_0}(\R^n)}.
\end{equation}
Because of $1-(\frac{1}{p_0}-\frac{1}{q_0})=
\frac{(m+2)n-2}{(m+2)n+2}$, it follows from the
Hardy-Littlewood-Sobolev inequality that
\begin{multline}\label{equ:3.33}
\| (AA^*)_\lambda G\|_{L^{q_0}(\R^{n+1}_+)}
=\left\| \int_0^{\infty}T_{t,\tau}G\,d\tau \right\|_{L^{q_0}(\R^{n+1}_+)}\\
\begin{aligned}
&\leq C\lambda^{-2s-\frac{2m}{(m+2)((m+2)n+2)}}\lambda^{\f{2(n+1)}{(m+2)n+2}}\left\|\int_{\R}
  |t-\tau|^{-\frac{(m+2)n-2}{(m+2)n+2}}\left\| G(\tau,\cdot)\right\|_{L^{p_0}(\R^n)}\,
  d\tau\right\|_{L^{p_0}(\R)} \\
%&\leq
%  C\lambda^{-2s-\frac{2m}{(m+2)((m+2)n+2)}}\lambda^{\f{2(n+1)}{(m+2)n+2}}\|
%  G\|_{L^{p_0}(\R^{n+1}_+)} \\
&\leq C\lambda^{-2s+\frac{2}{m+2}}\left\| G\right\|_{L^{p_0}(\R^{n+1}_+)}.
\end{aligned}
\end{multline}
By interpolation between \eqref{equ:3.27} and \eqref{equ:3.33}, we
have that, for $1\leq p\leq p_0$,
\[
\|(AA^*)_\lambda G\|_{L^{q}(\R^{n+1}_+)}\leq
C\lambda^{-2s+2\left(\frac{n}{2}-\frac{(m+2)n+2}{(m+2)q}\right)}\|
G\|_{L^p(\R^{n+1}_+)}.
\]
In particular, choosing $s=\frac{n}{2}-\frac{(m+2)n+2}{(m+2)q}$
yields estimate \eqref{equ:3.18} for $v_1(t,x)$. The same estimate for
$v_2(t,x)$ is analogously obtained.

Thus, the proof of Lemma~\ref{lem3.3} is complete.
\end{proof}

Next we treat the inhomogeneous problem \eqref{equ:3.3}. Based on
Lemmas~\ref{lem3.1} and \ref{lem3.2}, we establish the following
estimate:

\begin{lemma}\label{lem3.4}
Let $n\geq 2$ and $w$ solve \eqref{equ:3.3}. Then
\begin{equation}\label{equ:3.34}
\| w\|_{L^q(\R^{n+1}_+)}\leq C\,\bigl\||D_x|^{\gamma-\frac{1}{m+2}}
F\bigr\|_{L^{p_0}(\R^{n+1}_+)},
\end{equation}
where $\gamma=\frac{n}{2}-\frac{(m+2)n+2}{q(m+2)}$, $q_0\leq
q<\infty$, and the constant $C>0$ only depends on $m$, $n$ and $q$.
\end{lemma}

\begin{proof}
It follows from problem \eqref{equ:3.3} that
\[
w(t,x)=\int_0^t\left(V_2(t, D_x)V_1(\tau, D_x)-V_1(t, D_x)V_2(\tau,
D_x)\right)F(\tau,x)\,d\tau.
\]
To estimate $w(t,x)$, it suffices to treat the term $\int_0^tV_2(t,
D_x)V_1(\tau, D_x)F(\tau,x)d\tau$ since the treatment on the term
$\int_0^tV_1(t, D_x)V_2(\tau, D_x)F(\tau,x)d\tau$ is completely
analogous. Choose a cut-off function $\chi$ as in
\eqref{equ:3.11}. Set
\begin{align*}
w_1(t,x)&=\int_0^t\chi(\phi(t)D_x)\chi(\phi(\tau)D_x)V_2(t,
D_x)V_1(\tau, D_x)F(\tau,x)\,d\tau, \\
w_2(t,x)&=\int_0^t\chi(\phi(t)D_x)\left(1-\chi(\phi(\tau)D_x)\right)
V_2(t, D_x)V_1(\tau, D_x)F(\tau,x)\,d\tau, \\
w_3(t,x)&=\int_0^t\left(1-\chi(\phi(t)D_x)\right)\chi(\phi(\tau)D_x)V_2(t,
D_x)V_1(\tau, D_x)F(\tau,x)\,d\tau, \\
w_4(t,x)&=\int_0^t\left(1-\chi(\phi(t)D_x)\right)
\left(1-\chi(\phi(\tau)D_x)\right)V_2(t, D_x)V_1(\tau, D_x)F(\tau,x)\,d\tau.
\end{align*}
Together with
\eqref{equ:3.7}-\eqref{equ:3.10}, as in the proof of
Lemma~\ref{lem3.3}, we can write $\sum_{j=1}^{4}w_j$ as
\begin{equation}\label{equ:3.36}
\sum_{j=1}^{4}w_j=(AF)(t,x)\equiv
\int_0^t\int_{\R^n}e^{i\left(x\cdot\xi+(\phi(t)-\phi(\tau))|\xi|\right)}
a(t,\tau,\xi)\hat{F}(\tau,\xi)\,d\xi d\tau,
\end{equation}
where $a(t,\tau,\xi)$ satisfies
\[
\bigl| \partial_\xi^\beta a(t,\xi)\bigr|\leq
C\left(1+\phi(t)|\xi|\right)^{-\f{m}{2(m+2)}}
\left(1+\phi(\tau)|\xi|\right)^{-\f{m}{2(m+2)}}|\xi|^{-\frac{2}{m+2}-|\beta|}.
\]
To treat $(AF)(t,x)$ conveniently, we introduce the more general
operator
\begin{equation}\label{equ:3.37}
(A^\alpha F)(t,x)=\int_0^t\int_{\R^n}e^{i\left(x\cdot\xi+(\phi(t)-\phi(\tau))|\xi|\right)}
a(t,\tau,\xi)\hat{F}(\tau,\xi)\frac{d\xi}{|\xi|^\alpha}\,d\tau,
\end{equation}
where $0\leq\alpha<\frac{n}{2}$ is a parameter.

As in the proof of Lemma~\ref{lem3.3}, we shall use the
Littlewood-Paley argument with a bump function $\beta$ as in
Lemma~\ref{lem3.1}.
%%\eqref{equ:3.24}.
Define the operator
\begin{equation}\label{equ:3.38}
  A^\alpha_j F(t,x)=\int_0^t\int_{\R^n}e^{i\left(x\cdot\xi+(\phi(t)-\phi(\tau))|\xi|\right)}
  \beta\left(\frac{|\xi|}{2^j}\right)a(t,\tau,\xi)
\hat{F}(\tau,\xi)\frac{d\xi}{|\xi|^\alpha} \, d\tau.
\end{equation}
Note that our aim is to establish the inequality, for $
\gamma=\frac{n}{2}-\frac{(m+2)n+2}{q(m+2)}$, $q_0\le q<\infty$,
\begin{equation*}
\| w\|_{L^q(\R^{n+1}_+)}\leq C\bigl\||D_x|^{\gamma-\frac{1}{m+2}}
F\bigr\|_{L^{p_0}}
\end{equation*}
which is equivalent to proving that
\begin{equation*}
\left\| |D_x|^{-\gamma+\frac{1}{m+2}}w\right\|_{L^q(\R^{n+1}_+)}\leq
C\left\| F\right\|_{L^{p_0}(\R^{n+1}_+)}.
\end{equation*}
In terms of the operator $A^\alpha$ in \eqref{equ:3.37} with
$\alpha=\gamma-\frac{1}{m+2}$, it suffices to establish
\begin{equation}\label{equ:3.39}
\| A^\alpha F\|_{L^q(\R^{n+1}_+)}\le C\left\| F\right\|_{L^{p_0}(\R^{n+1}_+)}.
\end{equation}
in order to complete the proof of \eqref{equ:3.34}.

Note that $p_0<2<q<\infty$. To derive \eqref{equ:3.39}, it follows
from Lemma~\ref{lem3.1} that we only need to prove
\begin{equation}\label{equ:3.40}
\| A^\alpha_j F\|_{L^q(\R^{n+1}_+)}\le C\left\| F\right\|_{L^{p_0}(\R^{n+1}_+)}.
\end{equation}
By interpolation, it suffices to prove that \eqref{equ:3.40} holds for
the special cases $q=q_0$ and $q=\infty$. Denote the corresponding
indices $\alpha$ by $\alpha_0$ and $\alpha_1$.  A direct computation
yields $\alpha_0=\frac{n}{2}-\frac{(m+2)n+2}{q_0(m+2)}
-\frac{1}{m+2}=0$ and $\al_1=\frac{n}{2}-\frac{1}{m+2}$. We now
treat $A^{\alpha_0}_j=A^{0}_j$. Let
\[
T_j^0G(t,\tau,x)=\int_{\R^n}
e^{i\left(x\cdot\xi+(\phi(t)-\phi(\tau))|\xi|\right)}
\beta\left(\frac{|\xi|}{2^j}\right)a(t,\tau,\xi)\hat{G}(\tau,\xi)\, d\xi.
\]
We can repeat the derivation of \eqref{equ:3.32} to get
\begin{equation}\label{equ:3.41}
\| T_j^0G(t,\tau,\cdot)\|_{L^{p_0'}(\R^n)}\leq
C\left|t-\tau\right|^{-\frac{(m+2)n-2}{(m+2)n+2}}\left\|
G(\tau,\cdot)\right\|_{L^{p_0}}.
\end{equation}
Note that $\ds A^{0}_j G(t,x)=\int_0^tT_j^0G(t,\tau,x)\,d\tau$. Then,
by \eqref{equ:3.41} and the Hardy-Littlewood-Sobolev inequality, we
get
\[
\left\|\int_{\R}\|
T_j^0G(t,\tau,x)\|_{L^{q_0}_x}\,d\tau\right\|_{L^{q_0}_t}\leq C\left\|
G\right\|_{L^{p_0}}.
\]
With
\begin{equation*}
K(t,\tau)=
\begin{cases}
|t-\tau|^{-\frac{(m+2)n-2}{(m+2)n+2}}, &\tau\geq0, \\
0, &\tau<0,
\end{cases}
\end{equation*}
it follows from Lemma~\ref{lem3.2} with $q=q_0$ that \eqref{equ:3.40}
has been obtained.

Next we prove \eqref{equ:3.40} for $q=\infty$.  In this case, the
kernel of $A^{\alpha_1}_j$ can be written as
\[
K^{\alpha_1}_j(t,x;\tau,y)=\int_{\R^n}\beta\left(\frac{|\xi|}{2^j}\right)
e^{i\left((x-y)\cdot\xi+(\phi(t)-\phi(\tau))|\xi|\right)}a(t,\tau,\xi)\,\frac{d\xi}{|\xi|^{\alpha_1}}.
\]
We now assert
\begin{equation}\label{equ:3.42}
\sup\limits_{t,x}\int_{\R^{n+1}_+}|K^{\alpha_1}_j(t,x;\tau,y)|^{q_0}\,d\tau dy<\infty.
\end{equation}
Obviously, if \eqref{equ:3.42} is true, then a direct application of
H\"{o}lder's inequality yields \eqref{equ:3.40} for
$q=\infty$.

Next we turn to the proof of \eqref{equ:3.42}. By
\cite[Lemma~7.2.4]{Sog}, we have
\begin{multline}\label{equ:3.43}
\left|K^{\alpha_1}_j(t,x;\tau,y)\right| \\ \leq \,
C_{N,n,\alpha_1}\lambda^{\frac{n+1}{2}-\bar{\alpha_1}}
\left(|\phi(t)-\phi(\tau)|+\lambda^{-1}\right)^{-\frac{n-1}{2}}
\left(1+\lambda\big||x-y|-|\phi(t)-\phi(\tau)|\big|\right)^{-N},
\end{multline}
where $\lambda=2^j$, $N=0,1,2,\ldots$, and
\[
\bar{\alpha_1}=\frac{2}{m+2}+\alpha_1=\frac{2}{m+2}+\frac{n}{2}-\frac{1}{m+2}
=\frac{n}{2}+\frac{1}{m+2}.
\]
It suffices to prove \eqref{equ:3.42} in case $x=0$. In fact, a direct
computation yields
\begin{multline*}
\int_{\R^{n+1}}|K^{\alpha_1}_j(t,0;\tau,y)|^{q_0}\,d\tau dy \\
\begin{aligned}
&\leq\int_{-\infty}^{\infty}\int_{\R^n}\lambda^{(\frac{n+1}{2}-\bar{\alpha_1})\cdot q_0}
  \left(|\phi(t)-\phi(\tau)|+\lambda^{-1}\right)^{-\frac{n-1}{2}\cdot q_0}
  \left(1+\lambda\bigl||y|-|\phi(t)-\phi(\tau)|\bigr|\right)^{-N}\,dsdy \\
&\leq C\int_{-\infty}^{\infty}\lambda^{\frac{m}{2(m+2)}\cdot q_0}
  \left(|\phi(t)-\phi(\tau)|+\lambda^{-1}\right)^{-\frac{n-1}{2}\cdot q_0}\lambda^{-1}
  \left(|\phi(t)-\phi(\tau)|+\lambda^{-1}\right)^{n-1}\,d\tau \\
&\leq C \int_{-\infty}^{\infty}\lambda^{\frac{m(m+2)n+2m}{(m+2)((m+2)n-2)}-1}
  \left(|t-\tau|+\lambda^{-\frac{2}{m+2}}\right)^{-\frac{2(n-1)(m+2)}{(m+2)n-2}}\,d\tau \\
&\leq C.
\end{aligned}
\end{multline*}
Thus, by interpolation, \eqref{equ:3.40} and then further \eqref{equ:3.34} are shown.
\end{proof}

Relying on Lemmas~\ref{lem3.3} and \ref{lem3.4}, we have:

\begin{lemma}\label{lem3.5}
Let $w$ solve \eqref{equ:3.3}. Then
\begin{equation}\label{equ:3.44}
  \| w\|_{L^q(\R^{n+1}_+)}+\left\||D_x|^{\gamma-\frac{1}{m+2}} w\right\|_{L^{q_0}(\R^{n+1}_+)}\leq
  C\left\||D_x|^{\gamma-\frac{1}{m+2}} F\right\|_{L^{p_0}(\R^{n+1}_+)},
\end{equation}
where $\gamma=\frac{n}{2}-\frac{(m+2)n+2}{q(m+2)}$, $q_0\leq q<\infty$,
and the constant $C$ only depends on $m$, $n$, and $q$.
\end{lemma}

\begin{proof}
Note that
\[
\left(\partial_t^2-t^m \Delta\right)|D_x|^{\gamma-\frac{1}{m+2}} w=|D_x|^{\gamma-\frac{1}{m+2}} F.
\]
Then applying Lemma~\ref{lem3.4} with $q=q_0$ yields
\[
\left\||D_x|^{\gamma-\frac{1}{m+2}} w\right\|_{L^{q_0}}\leq
C\left\||D_x|^{\gamma-\frac{1}{m+2}} F\right\|_{L^{p_0}}.
\]
Together with Lemma~\ref{lem3.3}, this gives \eqref{equ:3.44}.
\end{proof}

%% -------------------------------------------------------------------

\section{Proof of Theorem~\ref{thm1.2}}

Based on the results of Section~\ref{sec3}, here we shall prove
Theorem~\ref{thm1.2}.  To establish the existence of a global solution
of \eqref{equ:original}, we shall use the iteration scheme
\begin{equation}\label{equ:4.1}
\left\{ \enspace
\begin{aligned}
&\partial_t^2 u_k-t^m \Delta u_k =|u_{k-1}|^p,   \\
&u_k(0,\cdot)=u_0(x), \quad \partial_{t} u_k(0,\cdot)=u_1(x),
\end{aligned}
\right.
\end{equation}
where $u_{-1}\equiv0$.

\begin{proof}[Proof of Theorem~\ref{thm1.2}]

We divide the proof into two parts.

\subsubsection*{Part 1. \ \boldmath $p_{\conf}(m,n)\leq p\leq
  \frac{(m+2)(n-2)+6}{(m+2)(n-2)-2}$.}

We will show that there is a solution $u\in L^r(\R^{n+1}_+)$ of
\eqref{equ:original} with
$r=\left(\frac{m+2}{2}n+1\right)\frac{p-1}{2}$ such that
$u_k\rightarrow u$ and $|u_k|^p\rightarrow |u|^p$ in
$\mathcal{D}'(\R^{n+1}_+)$ as $k\to\infty$.

We have that $
\frac{1}{m+2}\leq\gamma=\frac{n}{2}-\frac{(m+2)n+2}{r(m+2)}\leq
1+\frac{1}{m+2}$ (using $r\ge q_0$). Set
\begin{equation}\label{equ:4.2}
  M_k=\| u_k\|_{L^r(\R^{n+1}_+)}
  +\left\||D_x|^{\gamma-\frac{1}{m+2}}u_k\right\|_{L^{q_0}(\R^{n+1}_+)}.
\end{equation}
Suppose that we have already shown that, for $l=1,2,\dots,k$,
\begin{equation}\label{equ:4.3}
M_l\leq 2M_0\leq C\epsilon_0.
\end{equation}
Then we prove that \eqref{equ:4.3} also holds for $l=k+1$. Applying
Lemma~\ref{lem3.4} to the equation
\[
\left(\partial_t^2-t^m \Delta\right)(u_{k+1}-u_0) =F(u_k),
\]
where $F(u_k)=|u_{k}|^p$, we arrive at
\begin{equation}\label{equ:4.4}
\begin{aligned}
  M_{k+1}&\leq C\,\bigl\||D_x|^{\gamma-\frac{1}{m+2}} (F(u_k))\bigr\|_{L^{p_0}(\R^{n+1}_+)}
  +M_0 \\
&\leq C\,\| F'(u_k)\|_{L^{\frac{(m+2)n+2}{4}}(\R^{n+1}_+)}
\bigl\||D_x|^{\gamma-\frac{1}{m+2}} u_k\bigr\|_{L^{q_0}(\R^{n+1}_+)}+M_0 \\
&\leq C\| F'(u_k)\|_{L^{\frac{(m+2)n+2}{4}}(\R^{n+1}_+)}M_k+M_0.
\end{aligned}
\end{equation}
We mention that in this computation the following Leibniz's rule for
fractional derivatives has been used (see \cite{Chr1, Chr2} for
details):
\begin{equation}\label{equ:4.5}
\bigl\||D_x|^{\gamma-\frac{1}{m+2}}F(u)(s,\cdot)\bigr\|_{L^{p_1}(\R^n)} \leq\|
F'(u)(s,\cdot)\|_{L^{p_2}(\R^n)}\||D_x|^{\gamma-\frac{1}{m+2}}u(s,\cdot)\|_{L^{p_3}(\R^n)},
\end{equation}
where $\frac{1}{p_1}=\frac{1}{p_2}+\frac{1}{p_3}$ with $p_i\ge 1$
($1\le i\le 3$) and $0\leq\gamma-\frac{1}{m+2}\leq 1$. Moreover, it
follows from H\"{o}lder's inequality that
\begin{equation}\label{equ:4.6}
\| F'(u_k)\|_{L^{\frac{(m+2)n+2}{4}}(\R^{n+1}_+)}
\leq C\,\| u_k\|_{L^r(\R^{n+1}_+)}^{p-1}\leq CM_k^{p-1}\leq C(2M_0)^{p-1}.
\end{equation}
Thus, if $M_0\le C\epsilon_0$ and $\epsilon_0$ is so small that
\[
C(2M_0)^{p-1}\leq\tilde{C}\epsilon_0^{p-1}\leq\frac{1}{2},
\]
then we have
\[
M_{k+1}\leq\frac{1}{2}M_k+M_0\leq 2M_0.
\]
Next we estimate $M_0$. By Lemma~\ref{lem3.3}, we have that
\begin{equation}\label{equ:4.7}
M_0\leq C\left(\| f\|_{\dot{H}^s(\R^n)}+\| g\|_{\dot{H}^{s-\frac{2}{m+2}}(\R^n)}\right)
\le C\epsilon_0,
\end{equation}
where $s=\frac{n}{2}-\frac{(m+2)n+2}{(m+2)r}$ and $q_0\leq
r<\infty$. Therefore, we have obtained the uniform boundedness of
$\{M_k\}$.

Next we show that the sequence $\{u_k\}$ is convergent under the
weaker norm $\| \cdot\|_{L^{q_0}(\R^{n+1}_+)}$.  Set $N_k=\|
u_k-u_{k-1}\|_{L^{q_0}(\R^{n+1}_+)}$. Then
\begin{equation}\label{equ:4.8}
\begin{aligned}
N_{k+1}&=\| u_{k+1}-u_{k}\|_{L^{q_0}(\R^{n+1}_+)} \leq\| F(u_{k})-F(u_{k-1})\|_{L^{p_0}(\R^{n+1}_+)} \\
&\leq \bigl(\| u_k\|_{L^r(\R^{n+1}_+)}+\| u_{k-1}\|_{L^r(\R^{n+1}_+)}\bigr)^{p-1}
\| u_{k}-u_{k-1}\|_{L^{q_0}(\R^{n+1}_+)} \\
&\leq(M_k+M_{k-1})^{p-1}\| u_{k}-u_{k-1}\|_{L^{q_0}(\R^{n+1}_+)}
\leq C\epsilon_0^{p-1}\| u_{k}-u_{k-1}\|_{L^{q_0}(\R^{n+1}_+)} \\
&\leq\frac{1}{2}\,\| u_{k}-u_{k-1}\|_{L^{q_0}(\R^{n+1}_+)}
=\frac{1}{2}N_k
\end{aligned}
\end{equation}
Therefore, $u_k\rightarrow u$ in $L^{q_0}(\R^{n+1}_+)$ and hence in
$\mathcal{D}'(\R^{n+1}_+)$. This yields that there exists a subsequence, which is
still denoted by $\{u_k\}$, such that $u_k\rightarrow u$ a.e.  In
view of $\| u_k\|_{L^r(\R^{n+1}_+)}\le 2M_0$, it follows from Fatou's
lemma that
\[
\| u\|_{L^r(\R^{n+1}_+)}\leq\liminf_{k\rightarrow\infty}
\| u_k\|_{L^r(\R^{n+1}_+)}\leq2M_0\leq C\epsilon_0<\infty.
\]
It remains to prove that $F(u_k)\rightarrow F(u)$ in
$\mathcal{D}'(\R^{n+1}_+)$ in order to show that $u$ is a solution of
\eqref{equ:original}. In fact, for any fixed compact set
$K\Subset\R^{n+1}_+$, one has
\begin{equation}\label{equ:4.9}
\begin{aligned}
\| F(u_k)-F(u)\|_{L^1(K)}
& \leq C_K\| F(u_k)-F(u)\|_{L^{p_0}(K)} \\
& \leq C_K(\| u_k\|_{L^r(\R^{n+1}_+)}+\| u\|_{L^r(K)})^{p-1}\| u_k-u\|_{L^{q_0}(K)} \\
\leq&\tilde{C_K}\epsilon_0^{p-1}\| u_k-u\|_{L^{q_0}(K)}\rightarrow 0 \quad as \quad k\rightarrow\infty.
\end{aligned}
\end{equation}
Thus $|u_k|^p\rightarrow |u|^p$ in $L^1_{loc}(\R^{n+1}_+)$ and hence in $\mathcal{D}'(\R^{n+1}_+)$.

The proof of Part 1 is complete.

\subsubsection*{Part 2. \ \boldmath $p\ge\frac{(m+2)(n-2)+6}{(m+2)(n-2)-2}$,
  $p$ is an integer, and $|u^p|$ in \eqref{equ:original} is replaced
  with $\pm u^p$.}

We will show that there is a solution $u\in L^r(\R^{n+1}_+)$
of \eqref{equ:original} with $
r=\left(\frac{m+2}{2}n+1\right)\frac{p-1}{2}$ such that
$u_k\rightarrow u$ and $u_k^p\rightarrow u^p$ in
$\mathcal{D}'(\R^{n+1}_+)$ as $k\to\infty$.

We have that $\gamma=\frac{n}{2}-
\frac{(m+2)n+2}{(m+2)r}>1+\frac{1}{m+2}$. Let
\begin{equation}\label{equ:4.10}
M_k=\sup\limits_{q_0\leq q\leq r}\left\||D_x|^{\frac{(m+2)n+2}{q(m+2)}-\frac{2}{m+2}
\cdot\frac{2}{p-1}}u_k\right\|_{L^q(\R^{n+1}_+)}.
\end{equation}
Applying Lemma~\ref{lem3.4} to the equation
\[
\left(\partial_t^2-t^m \Delta\right)(u_{k+1}-u_0)=|u_k|^p
\]
yields
\begin{equation}\label{equ:4.11}
M_{k+1}\leq M_0+C_p\left\||D_x|^{\frac{n}{2}-\frac{1}{m+2}-\frac{2}{m+2}
\cdot\frac{2}{p-1}}|u_k|^p\right\|_{L^{p_0}(\R^{n+1}_+)}.
\end{equation}
To treat the second summand on the right-hand side of \eqref{equ:4.11}, we need the
following variant of \eqref{equ:4.5} (see \cite{Kpv} for details):
\begin{equation} \label{equ:4.12}
\left\||D_x|^\sigma (fg)\right\|_{L^p}\leq C\left\||D_x|^\sigma f\right\|_{L^{r_1}}
\| g\|_{L^{r_2}}+C\| f\|_{L^{s_1}}\left\||D_x|^\sigma g\right\|_{L^{s_2}},
\end{equation}
where $0\leq\sigma\leq1$, $1<r_j, s_j<\infty$, and $
\frac{1}{p}=\frac{1}{r_1}+\frac{1}{r_2}=\frac{1}{s_1}+\frac{1}{s_2}$.

By \eqref{equ:4.12} together with the fact that, for a given
multi-index $\alpha$ and $1<p<\infty$,
\[
\left\| D_x^\alpha f\right\|_{L^p}\leq
C_{p,\alpha}\left\||D_x|^{|\alpha|} f\right\|_{L^p},
\]
we arrive at
\[
\left\||D_x|^{\frac{n}{2}-\frac{1}{m+2}-\frac{2}{m+2}
  \,\frac{2}{p-1}}\left(|u_k|^p\right)\right\|_{L^{p_0}(\R^{n+1}_+)}
\le
C\,\prod\limits^p_{j=1}\left\||D_x|^{\alpha_j}u_k
\right\|_{L^{q_j}(\R^{n+1}_+)},
\]
where $0\leq\alpha_j\leq\frac{n}{2}-\frac{1}{m+2}-\frac{2}{m+2}
\,\frac{2}{p-1}$ and
\begin{equation}\label{equ:4.13}
  \sum\limits^p_{j=1}\alpha_j=
  \frac{n}{2}-\frac{1}{m+2}-\frac{2}{m+2}\,\frac{2}{p-1}.
\end{equation}
Let $q_0\leq q_j<\infty$ satisfy
\begin{equation}\label{equ:4.14}
\sum\limits^p_{j=1}\frac{1}{q_j}=\frac{1}{p_0},
\end{equation}
where $q_j$ is determined by
\[
\frac{(m+2)n+2}{q_j(m+2)}-\frac{2}{m+2}\,\frac{2}{p-1}=\alpha_j.
\]
From this, we have
\[
q_0\leq q_j\leq\frac{(m+2)n+2}{4}\left(p-1\right)
\]
and
\begin{equation}\label{equ:4.15}
\begin{aligned}
\sum\limits^p_{j=1}\frac{1}{q_j}&=\frac{m+2}{(m+2)n+2}\sum\limits^p_{j=1}
\left(\alpha_j +\frac{2}{m+2}\cdot\frac{2}{p-1}\right)\\
&=\frac{m+2}{(m+2)n+2}\left(\frac{n}{2}-\frac{1}{m+2}
-\frac{2}{m+2}\,\frac{2}{p-1}
+\frac{2p}{m+2}\,\frac{2}{p-1}\right) \\
&=\frac{1}{p_0}.
\end{aligned}
\end{equation}
Thus one has from \eqref{equ:4.11} that
\[
M_{k+1}\leq M_0+C_pM_k^p.
\]
Suppose that $M_k\leq2M_0\leq C\epsilon_0$ holds. Then
\[
M_{k+1}\leq M_0+C_p(2M_0)^{p-1}M_k\leq M_0+\tilde{C}_p\epsilon_0^{p-1}M_k.
\]
If $\epsilon_0>0$ is so small that $\tilde{C}_p\epsilon_0^{p-1}\leq1/2$, then
\[
M_{k+1}\leq M_0+\frac{1}{2}M_k\leq M_0+\frac{1}{2}\cdot2M_0=2M_0.
\]
Thus, we have obtain the uniform boundedness of the $M_k$
provided that $M_0\leq C\epsilon_0$.

Furthermore, we then have that, if $N_k$ is defined as in \eqref{equ:4.8},
\begin{equation*}
\begin{aligned}
N_{k+1}&=\| u_{k+1}-u_{k}\|_{L^{q_0}(\R^{n+1}_+)} \\
&\leq\left\| |u_{k}|^p- |u_{k-1}|^p\right\|_{L^{p_0}(\R^{n+1}_+)} \\
&\leq\left(\| u_k\|_{L^r(\R^{n+1}_+)}+\| u_{k-1}\|_{L^r(\R^{n+1}_+)}
\right)^{p-1}\| u_{k}-u_{k-1}\|_{L^{q_0}(\R^{n+1}_+)} \\
&\leq\left(\sup\limits_{q_0\leq q\leq r}\left\||D_x|^{\frac{(m+2)n+2}{q(m+2)}-\frac{2}{m+2}
\cdot\frac{2}{p-1}}u_k\right\|_{L^q(\R^{n+1}_+)} \right. \\
&\left. \qquad +\sup\limits_{q_0\leq q\leq r}\left\||D_x|^{\frac{(m+2)n+2}{q(m+2)}
  -\frac{2}{m+2}\cdot\frac{2}{p-1}}u_{k-1}\right\|_{L^q(\R^{n+1}_+)}\right)^{p-1}
\| u_{k}-u_{k-1}\|_{L^{q_0}(\R^{n+1}_+)} \\
&\leq\left(M_k+M_{k-1}\right)^{p-1}\| u_{k}-u_{k-1}\|_{L^{q_0}(\R^{n+1}_+)} \\
&\leq C\epsilon_0^{p-1}\,\| u_{k}-u_{k-1}\|_{L^{q_0}(\R^{n+1}_+)} \\
&\leq\frac{1}{2}\,\| u_{k}-u_{k-1}\|_{L^{q_0}(\R^{n+1}_+)} =\frac{1}{2}N_k.
\end{aligned}
\end{equation*}
Thus, $u_k\rightarrow u$ in $L^{q_0}(\R^{n+1}_+)$ as $k\to\infty$.
From here we can finish the proof of Part 2 as in Part 1.

\bigskip

Part~1 and Part~2 jointly constitute the proof of Theorem~\ref{thm1.2}.
\end{proof}

%% -------------------------------------------------------------------

%% -------------------------------------------------------------------


\begin{thebibliography}{99}

\bibitem{Bar} J.~Barros-Neto and I.~M.~Gelfand, \textit{Fundamental
  solutions for the Tricomi operator. I, II, III}, Duke
  Math. J.~\textbf{98} (1999), 465--483; \textbf{111} (2002), 561--584
  (2002); \textbf{117} (2003), 385--387.

\bibitem{Chr1} M.~Christ, \textit{Lectures on Singular Integral
  Operators}, CBMS Regional Conf. Ser. in Math. , vol. 77, Amer. Math. Soc.,
  Providence, RI, 1990.

\bibitem{Chr2} M.~Christ and M.~Weinstein, \textit{Dispersion of
  low-amplitude solutions of generalized Korteweg-de Vries equation},
  J. Funct. Anal. \textbf{100} (1991), 87--109.

\bibitem{Rei2} M.~D'Abbicco, S.~Lucente, and M.~Reissig,
  \textit{Semilinear wave equations with effective damping}, Chin.
  Ann. Math. Ser.~B \textbf{34} (2013), 345--380.

\bibitem{Rei1} M.~D'Abbicco, S.~Lucente, and M.~Reissig, \textit{A
  shift in the Strauss exponent for semilinear wave equations with a
  not effective damping}, J. Differential Equations \textbf{259} (2015),
  5040--5073.

\bibitem{Erd1} A.~Erdelyi, W.~Magnus, F.~Oberhettinger, and
  F.G.~Tricomi, \textit{Higher Transcendental Functions, Vol.~1},
  McGraw-Hill, New York, 1953.

\bibitem{Erd2} A.~Erdelyi, W.~Magnus, F.~Oberhettinger, and
  F.G.~Tricomi, \textit{Higher Transcendental Functions, Vol.~2},
  McGraw-Hill, New York, 1953.

\bibitem{Fuj} H.~Fujita, \textit{On the blowing up of solutions of the
  Cauchy Problem for $u_t=\Delta u+u^{1+\al}$},
  J. Fac. Sci. Univ. Tokyo \textbf{13} (1966), 109--124.

\bibitem{Gls1} V.~Georgiev, H.~Lindblad, and C.D.~Sogge,
  \textit{Weighted Strichartz estimates and global existence for
    semi-linear wave equations}, Amer. J. Math. \textbf{119} (1997),
  1291--1319.

\bibitem{Gv} J.~Ginibre and G.~Velo, \textit{The global Cauchy problem
  for the non-linear Klein-Gordan equation}, Math. Z. \textbf{189}
  (1985), 487--505.

\bibitem{Gla1} R.T.~Glassey, \textit{Finite-time blow-up for solutions
  of nonlinear wave equations}, Math. Z. \textbf{177} (1981),
  323--340.

\bibitem{Gla2} R.T.~Glassey, \textit{Existence in the large for
  $\Box$u =F(u) in two space dimensions}, Math. Z. \textbf{178}
  (1981), 233--261.

\bibitem{He} He Daoyin, I.~Witt, and Yin Huicheng, \textit{On the
  global solution problem of semilinear generalized Tricomi equations,
  II}, Preprint, 2015.

\bibitem{Hon} Hong Jiaxing and Li Goquan, \textit{$L^p$ estimates for
  a class of integral operators}, J. Partial Differ. Equ. \textbf{9}
  (1996), 343--364.

\bibitem{Joh} F.~John, {\it Blow-up of solutions of nonlinear wave
  equations in three space dimensions}, Manuscripta Math. \textbf{28}
  (1979), 235--265.

\bibitem{Kpv} C.E.~Kenig, G.~Ponce, and L.~Vega,
  \textit{Well-posedness and scattering results for the generalized
    Korteweg-de Vries equation via the contraction principle},
  Comm. Pure Appl. Math. \textbf{46} (1993), 527--620.

\bibitem{Zhai} Lin Jiayun, K.~Nishihara, and Zhai Jian,
  \textit{Critical exponent for the semilinear wave equation with
    time-dependent damping}, Discrete Contin. Dyn. Syst. \textbf{32}
  (2012), 4307--4320.


\bibitem{Gls2} H.~Lindblad and C.D.~Sogge, \textit{On existence and
  scattering with minimal regularity for semilinear wave equations},
  J. Funct. Anal. \textbf{130} (1995), 357--426.

  \bibitem {Nish} K. Nishihara, \textit{Asymptotic behavior of solutions to the semilinear wave equation with time-dependent damping},
Tokyo J. Math. \textbf{34} (2011), 327--343.


\bibitem{Rei3} M.~Reissig and J.~Wirth, \textit{$L^p$-$L^q$ decay
  estimates for wave equations with monotone time-dependent
  dissipation}, Preprint, Institute for Applied Analysis, TU
  Bergakademie Freiberg, 2008.

\bibitem{Rua2} Ruan Zhuoping, I.~Witt, and Yin Huicheng, \textit{The
  existence and singularity structures of low regularity solutions to
  higher order degenerate hyperbolic equations}, J. Differential
  Equations \textbf{256} (2014), 407--460.

\bibitem{Rua1} Ruan Zhuoping, I.~Witt, and Yin Huicheng, \textit{On
  the existence and cusp singularity of solutions to semilinear
  generalized Tricomi equations with discontinuous initial data},
  Commun. Contemp. Math., \textbf{17} (2015), 1450028 (49 pages).

\bibitem{Rua3} Ruan Zhuoping, I.~Witt, and Yin Huicheng, \textit{On
  the existence of low regularity solutions to semilinear generalized
  Tricomi equations in mixed type domains}, J. Differential Equations
  \textbf{259} (2015), 7406--7462.

\bibitem{Rua4} Ruan Zhuoping, I.~Witt, and Yin Huicheng, \textit{On
  the existence of solutions with minimal regularity for semilinear
  generalized Tricomi equations}, Preprint, 2015.

\bibitem{Sch} J.~Schaeffer, \textit{The equation $\Box u =|u|^p$ for
  the critical value of $p$}, Proc. Roy. Soc. Edinburgh \textbf{101}
  (1985), 31--44.

\bibitem{Sid} T.~Sideris, \textit{Nonexistence of global solutions to
  semilinear wave equations in high dimensions}, J. Differential
  Equations \textbf{52} (1984), 378--406.

\bibitem{Sog} C. D. Sogge, \textit{Fourier Integrals in Classical
  Analysis}, Cambridge Tracts in Math., vol.~105, Cambridge
  Univ. Press, Cambridge, 1993.

\bibitem{Strauss} W.~Strauss, \textit{Nonlinear scattering theory at
  low energy}, J. Funct. Anal. \textbf{41} (1981), 110--133.

\bibitem{Yag1} K.~Yagdjian, \textit{A note on the fundamental solution
  for the Tricomi-type equation in the hyperbolic domain},
  J. Differential Equations \textbf{206} (2004), 227--252.

\bibitem{Yag2} K.~Yagdjian, \textit{Global existence for the
  n-dimensional semilinear Tricomi-type equations}, Comm. Partial
  Diff. Equations \textbf{31} (2006), 907--944.

\bibitem{Yor} B.~Yordanov and Q.-S.~Zhang, \textit{Finite time blow up
  for critical wave equations in high dimensions},
  J. Funct. Anal. \textbf{231} (2006), 361--374.

\bibitem{Zhou} Zhou Yi, \textit{Cauchy problem for semilinear wave
  equations in four space dimensions with small initial data},
  J. Differential Equations \textbf{8} (1995), 135--144.

\end{thebibliography}
\end{document}